\documentclass[article]{IEEEtran}

\usepackage[T1]{fontenc}

\usepackage{cite}
\usepackage{amsmath,amssymb,amsfonts, amsthm}
\usepackage{xcolor} 

\usepackage{textcomp}

\usepackage{amssymb,latexsym,amsmath}
\usepackage{mathrsfs}
\usepackage{enumitem}
\usepackage{caption}
\usepackage{inputenc}

\usepackage{graphics} 
\usepackage{epsfig} 
\usepackage{verbatim}

\usepackage{amsfonts}
\usepackage{float}
\usepackage{multirow}
\usepackage{tikz}

\usetikzlibrary{shapes.geometric, arrows}
\allowdisplaybreaks

\newtheorem{definition}{Definition}[section]%
\newtheorem{theorem}[definition]{Theorem}%
\newtheorem{lemma}[definition]{Lemma}%
\newtheorem{assumption}[definition]{Assumption}%
\newtheorem{corollary}[definition]{Corollary}%
\newtheorem{remark}[definition]{Remark:}%

\IEEEoverridecommandlockouts

\allowdisplaybreaks

\usepackage{algorithmic}
\usepackage{epstopdf}

\usepackage{tikz}
\usetikzlibrary{shapes.geometric, arrows}
\allowdisplaybreaks

\usepackage{color}
\newcommand{\sy}[1]{{\color{black} #1}}

\tikzstyle{terminal}=[rectangle, rounded corners, minimum width=1cm, minimum height=1cm,text centered, draw=black]
\tikzstyle{arrow} = [thick,->,>=stealth]
\tikzstyle{addblock} = [draw,circle]

\begin{document}



\title{Stochastic Observability and Filter Stability under Several Criteria}


\author{Curtis McDonald and Serdar Y\"{u}ksel\thanks{C. McDonald is with the Department of Statistics and Data Science at Yale University, United States of America (e-mail: curtis.mcdonald@yale.edu). S.~Y\"uksel is with the Department of Mathematics and Statistics, Queen's University, Kingston, Ontario, Canada, K7L 3N6 (e-mail: yuksel@mast.queensu.ca). Research is supported by the Natural Sciences and Engineering Research Council of Canada. Some preliminary results of this submission were presented at the 2018 Annual Allerton Conference.}}%
\date{}%
\maketitle%

\begin{abstract}
Despite being a foundational concept of modern systems theory, there have been few studies on observability of non-linear stochastic systems under partial observations. In this paper, we introduce a definition of observability for stochastic non-linear dynamical systems which involves an explicit functional characterization. To justify its operational use, we establish that this definition implies filter stability under mild continuity conditions: an incorrectly initialized non-linear filter is said to be stable if the filter eventually corrects itself with the arrival of new measurement information. Numerous examples are presented and a detailed comparison with the literature is reported. We also establish implications for various criteria for filter stability under several notions of convergence such as weak convergence, total variation, and relative entropy. These findings are connected to robustness and approximations in partially observed stochastic control. 
\end{abstract}






\section{Introduction}\label{intro}


%

Observability is one of the most important and foundational concepts of modern systems and control theory with implications at the heart of its theory and applications \cite{Caines,Kalman,kushner2014partial, KushnerKalmanFilter}. For deterministic linear systems, observability is defined as the exact recovery of any initial condition with measurements available until some finite time, and is characterized by an observability rank condition in both continuous and discrete-time \cite{Chen}. For linear systems, such an observability definition is global (as it applies for all initial states) and is also directly applicable to stochastic counterparts of deterministic linear systems. For non-linear systems, however, due to the challenges in the analysis which prevent globality, the analysis is significantly more nuanced both for deterministic and stochastic setups. See Section \ref{lit_review} for a detailed discussion.

We study the stochastic setup in this paper. Let us now introduce the probabilistic setup for a Hidden Markov Model (HMM) or Partially Observed Markov Process (POMP). Let $(\mathcal{X},\mathcal{Y})$ be complete, separable and metric (Polish) spaces equipped with their Borel sigma fields $\mathcal{B}(\mathcal{X})$ and $\mathcal{B}(\mathcal{Y})$. $\mathcal{X}$ will be called the state space, and $\mathcal{Y}$ the measurement space. Let $\mathcal{P}(\mathcal{X})$ and $\mathcal{P}(\mathcal{Y})$ be the set of probability measures on these spaces. Define the transition kernel $T$ and measurement channel $G$ as the mappings
\begin{align*}
T:&\mathcal{X} \to \mathcal{P}(\mathcal{X})
&G:\mathcal{X} \to \mathcal{P}(\mathcal{Y})\\
&x \mapsto T(dx'|x)& x \mapsto G(dy|x)
\end{align*}
 The system is initialized with a state $X_{0} \in \mathcal{X}$ drawn from a prior measure $\mu$ on $(\mathcal{X},\mathcal{B}(\mathcal{X}))$. The state is then randomly updated via the transition kernel $T$ which makes the state process $\{X_{n}\}_{n=0}^{\infty}$ a Markov Chain with initial measure $\mu$ and transition kernel $T$.
 
However, the state is not available at the observer, instead at time $n$ the observer sees the observation $Y_{n}$ where the conditional distribution of $Y_{n}|X_{n}$ is determined by the measurement channel $G$. 

By stochastic realization arguments \cite[Lemma 1.2]{gihman2012controlled}, \cite[Lemma 3.1]{BorkarRealization}, we can also view an equivalent construction of the system dynamics. Let  $\{Z_{n}\}_{n=0}^{\infty}$ and $\{W_{n}\}_{n=0}^{\infty}$ be independent  identically distributed (i.i.d)  $\mathcal{Z}$-valued noise processes, where $\mathcal{Z}$ can be taken to be $[0,1]$ or $\mathbb{R}$ (or any other Polish space), without any loss of generality. Consider a partially observed dynamical system with the following model. 
\begin{eqnarray}\label{dynamicsEqn}
 X_{n+1} &=& \sy{b}(X_n, W_n) \\
Y_n &=& h(X_n, Z_n),
\end{eqnarray}
where \sy{$W_n$ and $Z_n$} can be assumed to take values from $[0,1]$ or $\mathbb{R}$. Here $\sy{b}$ defines the system dynamics and defines a transition kernel $T$ for the Markov chain $X_n$. Assuming $Z_{n}$ has measure $Q$ in $\mathcal{Z}$, the measurement function $h$ defines the measurement channel $G$ which is the pushforward measure of $Q$ under $h(x, \cdot)$. Throughout the paper we will work with either the general kernel and measurement channel notation $T, G$ or with the specific functional form using $\sy{b}, h$ when convenient.

 Thus, the observer needs to compute the conditional probability on the hidden variable $X_n$ using the information available up to time $n \in \mathbb{Z}_+$. We have that $\{X_{n},Y_{n}\}_{n=0}^{\infty}$ is a Markov chain, and we will denote $P^{\mu}$ as the probability measure on $\Omega=\mathcal{X}^{\mathbb{Z}_{+}}\times \mathcal{Y}^{\mathbb{Z}_{+}}$, endowed with the product topology, (and thus $\omega \in \Omega$ is a sequence of states and measurements $\omega=\{(x_{i},y_{i})\}_{i=0}^{\infty}$) where $X_{0} \sim \mu$. Such a stochastic system is referred to as a Partially Observed Markov Process (POMP) (also called Hidden Markov Model) throughout the paper.
\begin{definition} We define the {\it one step predictor} as the sequence
\[\pi_{n-}^{\mu}(\cdot)=P^{\mu}(X_{n} \in \cdot|Y_{0},...,Y_{n-1}),\quad n \in \mathbb{Z}_+\]
and we define the {\it non-linear filter} as the sequence
\[\pi_{n}^{\mu}(\cdot)=P^{\mu}(X_{n} \in \cdot|Y_{0},...,Y_{n}),\quad n \in \mathbb{Z}_+.\]
\end{definition}

Both of the above are regular conditional probability sequences defined on $\mathcal{X}$. We will use the notation $Y_{[0,n]}=Y_{0},\cdots, Y_{n}$ to represent finite sets of random variables, and $Y_{[0,\infty)}=Y_{0},Y_{1},\cdots$ to represent infinite sequences. The recursive update equations for the filter or the predictor are known as the non-linear filtering equations.
Let us, for the time being, assume the existence of a likelihood function $g(x,y)$ for the measurement channel defined as follows. The measurement channel $G$ is called dominated if there exists a reference measure $\lambda$ such that $\forall x \in \mathcal{X}, G(Y_{n} \in \cdot|X_{n}=x) \ll \lambda$ where the notation ``$\ll$'' denotes absolutely continuity. We can then utilize a likelihood function $g(x,y)=\frac{dG(Y_{n} \in \cdot|X_{n}=x)}{d\lambda}(y)$ and write the filter $\pi_{n+1}^{\mu}$ recursively in terms of $\pi_{n}^{\mu}$ and $Y_{n+1}=y_{n+1}$ explicitly as a Bayesian update:
\begin{align}
&\pi_{n+1}^{\mu}(dx_{n+1})=F(\pi_{n}^{\mu},y_{n+1})(dx_{n+1})  \nonumber \\
&:=\frac{g(x_{n+1},y_{n+1})\int_{\mathcal{X}}T(dx_{n+1}|X_{n}=x)\pi_{n}^{\mu}(dx)}{\int_{\mathcal{X}}g(x_{n+1},y_{n+1})\int_{\mathcal{X}}T(dx_{n+1}|X_{n}=x)\pi_{n}^{\mu}(dx)}
\end{align}


Suppose that an observer runs a non-linear filter assuming that the initial prior is $\nu$, when in reality the prior distribution is $\mu$. The observer receives the measurements and computes the filter $\pi_{n}^{\nu}$ for each $n$, but the measurement process is generated according to the true measure $\mu$. 
 
The operational question for observability is that of \textit{filter stability}, namely, if we have two different initial probability measures $\mu$ and $\nu$, when do we have that the filter processes $\pi_{n}^{\mu}$ and $\pi_{n}^{\nu}$ merge in some appropriate sense as $n \to \infty$. In essence, when will our observations $Y_{n}$ be informative enough to correct our incorrect prior $\nu$ and result in an accurate conditional measure for the hidden state.


In Section \ref{prelim} below, notations and definitions are presented. In Section \ref{main_res} we present our main results. We present a detailed literature review after the statement of our main results in Section \ref{lit_review}. Examples of observable systems are given in Section \ref{examples}. Proofs are provided in Section \ref{proof_sec}.



\subsection{Notation and preliminaries}\label{prelim}

Let $C_{b}(\mathcal{X})$ represent the set of continuous and bounded functions from $\mathcal{X} \to \mathbb{R}$.
\begin{definition}\label{weak_def}
Two sequences of probability measures $P_{n}$, $Q_{n}$ merge weakly if $\forall f \in C_{b}(\mathcal{X})$ we have $\lim_{n \to \infty} \left|\int f dP_{n}-\int f dQ_{n}\right|=0$.
\end{definition}
\begin{definition}\label{tv_def}
For two probability measures $P$ and $Q$ the total variation norm is defined as 
$\|P-Q\|_{TV}=\sup_{\|f\|_{\infty}\leq 1}\left|\int fdP-\int fdQ \right|$
where $f$ is assumed measurable.
\end{definition}
Note that merging in total variation implies weak merging since $C_{b}(\mathcal{X})$ is a subset of the set of measurable and bounded functions. We also utilize the relative entropy (Kullback-Leibler divergence) between two probability measures, although it is not a metric (since it is not symmetric).
\begin{definition}~
\begin{itemize}
\item[(i)] For two probability measures $P$ and $Q$ we define the relative entropy as
$
D(P\|Q)=\int \log \frac{dP}{dQ} dP=\int \frac{dP}{dQ}\log \frac{dP}{dQ}dQ
$
where $P \ll Q$ and $\frac{dP}{dQ}$ denotes the Radon-Nikodym derivative of $P$ with respect to $Q$.
\item[(ii)] Let $X$ and $Y$ be two random variables, let $P$ and $Q$ be two different joint measures for $(X,Y)$ with $P \ll Q$. We define the (conditional) relative entropy between $P(X|Y)$ and $Q(X|Y)$ as
\begin{align}
& D(P(X|Y)\|Q(X|Y)) =\int  \log \left(\frac{dP_{X|Y}}{dQ_{X|Y}}(x,y)\right)P(d(x,y)) \nonumber\\
&=\int \left(\int \log \left(\frac{dP_{X|Y}}{dQ_{X|Y}}(x,y)\right)P(dx|Y=y)\right)P(dy)\label{re_cond_def}
\end{align}
\end{itemize}
\end{definition}
Some notational discussion is in order. For some probability measures such as $P^{\mu}(Y_{[0,n]} \in \cdot)$ or $P^{\mu}(X_{n} \in \cdot)$, it will be convenient to denote the random variable inside the measure and take out the set argument.
When we take the relative entropy of such measures, to make the notation shorter, we will drop the ``$\in \cdot$'' argument and write $D(P^{\mu}(Y_{[0,n]})\|P^{\nu}(Y_{[0,n]}))$. 
 
Note that in a conditional relative entropy, we are integrating the logarithm of the Radon-Nikodym derivative of the conditional measures $P(X|Y)$ and $Q(X|Y)$ over the joint measure of $P$ on (X,Y). The second equality (\ref{re_cond_def}) shows that this can be thought of as the expectation of the relative entropy $D(P(X|Y=y)\|Q(X|Y=y))$ at specific realizations of $Y=y$ , where the expectation is over the marginal measure of $P$ on $Y$. When we apply this to the filter, $\pi_{n}^{\mu}$ and $\pi_{n}^{\nu}$ are realizations of the filter for specific measurements, therefore when we discuss their relative entropy, we take the expectation over the marginal of $P^{\mu}$ on $Y_{[0,n]}$. We write this as $E^{\mu}[D(\pi_{n}^{\mu}\|\pi_{n}^{\nu})]$ where $D(\pi_{n}^{\mu}\|\pi_{n}^{\nu})$ plays the role of the inner integral in (\ref{re_cond_def}).

We now introduce some additional notation that will be useful when dealing with sigma fields rather than random variables directly. Strictly speaking, we have two probability measures $P^{\mu}$ and $P^{\nu}$ on $(\mathcal{X}^{\mathbb{Z}_{+}}\times \mathcal{Y}^{\mathbb{Z}_{+}},\mathcal{B}(\mathcal{X}^{\mathbb{Z}_{+}} \times \mathcal{Y}^{\mathbb{Z}_{+}}))$. We denote by $\mathcal{F}^{\mathcal{X}}_{a,b}$ the sigma field generated by $(X_{a},\cdots, X_{b})$ and similarly for $\mathcal{Y}$. We also write $\mathcal{F}^{\mathcal{X}}_{n}$ for the sigma field generated by $X_{n}$. We then have $\mathcal{F}^{\mathcal{X}}_{0,\infty} \vee \mathcal{F}^{\mathcal{Y}}_{0,\infty}$ as the sigma field generated by all state and measurement sequences. When we write $P^{\mu}(X_{[0,n]})$ we are discussing the measure $P^{\mu}$ restricted to the sigma field $\mathcal{F}^{\mathcal{X}}_{0,n}$ which we will denote $P^{\mu}|_{\mathcal{F}^{\mathcal{X}}_{0,n}}$. Similarly for some set $A \in \mathcal{F}^{\mathcal{X}}_{0,\infty} \vee \mathcal{F}^{\mathcal{Y}}_{0,\infty}$ we write $P^{\mu}((X_{[0,\infty)},Y_{[0,\infty)}) \in A|Y_{[0,n]})$ as the conditional measure of $P^{\mu}$ with respect to the sigma field $\mathcal{F}^{\mathcal{Y}}_{0,n}$, which we denote $P^{\mu}|\mathcal{F}^{\mathcal{Y}}_{0,n}$. We can also consider restricting and conditioning simultaneously, this for example is the case with the non-linear filter:
$\pi_{n}^{\mu}(\cdot)=P^{\mu}(X_{n} \in \cdot|Y_{[0,n]})=P^{\mu}|_{\mathcal{F}^{\mathcal{X}}_{n}}|\mathcal{F}^{\mathcal{Y}}_{0,n}$.
The key relationship between relative entropy and total variation is Pinsker's inequality (see e.g., \cite{csiszar1967information}) which states that for two probability measures $P$ and $Q$ we have that $\|P-Q\|_{TV}\leq \sqrt{\frac{2}{\log(e)}D(P\|Q)}$.

\noindent{\bf Criteria for stability.}
We note the following definitions for filters, but they can also be defined for predictors.
\begin{definition}
\begin{itemize}
\item[(i)]  A filter process is stable in the sense of weak merging in expectation if for any $f \in C_{b}(\mathcal{X})$ and any prior $\nu$ with $\mu \ll \nu$ we have
\[\lim_{n \to \infty} E^{\mu}[ |\int f d\pi_{n}^{\mu}-\int f d\pi_{n}^{\nu} |]=0.\]
\item[(ii)] A filter process is stable in the sense of weak merging $P^{\mu}$ almost surely (a.s.) if there exists a set of measurement sequences $A \subset \mathcal{Y}^{\mathbb{Z}_{+}}$ with $P^{\mu}$ probability 1 such that for any sequence in $A$, for any $f \in C_{b}(\mathcal{X})$ and any prior $\nu$ with $\mu \ll \nu$ we have
\[\lim_{n \to \infty} |\int f d\pi_{n}^{\mu}-\int f d\pi_{n}^{\nu}|=0.\]

\item [(iii)]  A filter process is stable in the sense of total variation in expectation if for any measure $\nu$ with $\mu \ll \nu$ we have
\[\lim_{n \to \infty} E^{\mu}[\|\pi_{n}^{\mu}-\pi_{n}^{\nu}\|_{TV}]=0.\]

\item[(iv)]  A filter process is stable in the sense of total variation $P^{\mu}$ a.s. if for any measure $\nu$ with $\mu \ll \nu$ we have
\[\lim_{n \to \infty} \|\pi_{n}^{\mu}-\pi_{n}^{\nu}\|_{TV}=0~~P^{\mu}~a.s.\]

\item [(v)] A filter process is stable in relative entropy if for any measure $\nu$ with $\mu \ll \nu$:
\[\lim_{n \to \infty} E^{\mu}[D(\pi_{n}^{\mu}\|\pi_{n}^{\nu})]=0.\]

\item[(vi)] For $f:\mathcal{X} \to \mathbb{R}$ define the Lipschitz norm
\begin{align*}
\|f\|_{L}=\sup\left\{\left.\frac{|f(x)-f(y)|}{d(x,y)}\right|d(x,y)\neq 0\right\}
\end{align*}
 With $\text{BLip}:=\{f : \|f\|_{L}\leq 1, \|f\|_{\infty} \leq 1\}\subset C_{b}(\mathcal{X})$ we define the bounded Lipschitz (BL) metric as
\[\|P-Q\|_{BL}=\sup_{f \in \text{BLip}}\left|\int fdP -\int f dQ\right|.\]
A system is then stable in the sense of BL-merging $P^{\mu}$ a.s. if we have
$
\|\pi_{n}^{\mu}-\pi_{n}^{\nu}\|_{BL} \to 0~~~P^{\mu}~a.s.
$
\end{itemize}
\end{definition}

We note that \textit{merging} of probability measures is different from the \textit{convergence} of a sequence of probability measures to a limit measure. In convergence, we have some sequence $P_{n}$ and a static limit measure $P$; in merging we have two sequences $P_{n}$ and $Q_{n}$ which may not individually have limits, but come closer together for large $n$ in one of the merging notions defined previously \cite{d1988merging}

\section{Statement of the Main Results and Literature Review}\label{main_res}

\subsection{Stochastic non-linear observability}\label{obs_sec}
We first introduce our notion of an observable system.
\begin{definition}
\begin{itemize}
\item[(i)][One Step Observability]\label{one_step_obs}
A POMP is said to be one step observable if for every $f \in C_{b}(\mathcal{X}),\epsilon>0$, $\exists$ a measurable and bounded function $g:\mathcal{Y} \to \mathbb{R}$ such that
\begin{align*}
\left\|f(\cdot)- \int_{Y}g(y)G(dy|\cdot) \right\|_{\infty}<\epsilon
\end{align*}
\item[(ii)][$N$ Step Observability]\label{N_step_obs}
A POMP is said to be $N$ step observable if for every $f \in C_{b}(\mathcal{X}),\epsilon>0$, $\exists$ a measurable and bounded function $g:\mathcal{Y}^{N} \to \mathbb{R}$ such that
\begin{align}\label{Nstepobs}
\left\|f(\cdot)-\int_{\mathcal{Y}^{N}}g(y_{[1,N]})P(dy_{[1,N]}|X_{1}=\cdot)\right\|_{\infty}<\epsilon,
\end{align}
where we note that the conditional probability $P(dy_{[1,N]}|X_{1}=x_1)$ is independent of the prior measure.
\item[(iii)][Observability]\label{observable}
A POMP is {\it observable} if for every 
$f \in C_{b}(\mathcal{X})$ and every $\epsilon>0$ there exist $N \in \mathbb{N}$ and a measurable and bounded function $g$ (\sy{both possibly dependent on $f$ and $\epsilon$}) such that (\ref{Nstepobs}) applies. 
Note that if a POMP is $N$ step observable for some finite $N \in \mathbb{N}$ then it is {\it observable}, but the reverse implication is not necessarily the case.
\end{itemize}
\end{definition}

A number of remarks are in order.
\begin{remark}\label{Rm1}
In the definition above, we can instead of $C_{b}(\mathcal{X})$ consider any dense subset in $C_{b}(\mathcal{X})$. For example, if $\mathcal{X}$ is a compact subset of $\mathbb{R}^k$, we can consider polynomials as these form a dense subset, or we can consider smooth functions defined on $\mathcal{X}$, or functions which are expressed as linear combinations of harmonics, Haar wavelets etc.  An example is provided in Section \ref{example1}.
\end{remark}

\begin{remark}\label{Rm3}[One-step observability and universality in the controlled setup]
The definition of one step observability is a specific case of $N$ step observability, however the distinction is useful for at least two reasons: (i) One-step observability is often easier to check since one doesn't need to consider the effect of the state transition kernel, as this definition is only concerned with the measurement channel itself. On the other hand, there exist many setups where a system is observable, but not one step observable; see e.g. Section \ref{finiteEx}. (ii) Even though in this paper we consider a control-free setup, in a controlled context studied in a companion paper \cite{MYRobustControlledFS}, it follows that one step observability would be independent of any control policy (that is, observability would be universal over all policies and associated filter stability results apply under any control policy), but $N>1$ step observability would be handled much more cautiously as this condition would be dependent on the control policy adopted. Recently, filter stability results have been shown to be consequential in showing near-optimality of finite memory control policies and associated learning theoretic results for Partially Observed Markov Decision Processes \cite{kara2020near,kara2021convergence}. Accordingly, one-step observability results are particularly applicable for such scenarios. 
\end{remark}

\begin{remark}\label{Rm2}
If the measurement kernel satisfies an absolute continuity condition so that $G(dy|x) = h(x,y) \lambda(dy)$ and if there exists a finite measure $K$ such that $\sy{G}(dy|x) \leq K(dy)$ (so that the family of kernels $\{G(dy|x), x \in \mathcal{X}\}$ is majorized by $K$ leading to a uniformly countable additive family of measures), then by Lusin’s theorem \cite[Theorem 7.5.2]{Dud02} and the extension theorem of Tietze \cite[Theorem 4.1]{dugundji}, we can replace $g$ in the above with a continuous function $g_c$. The relaxation to such continuous $g_c$ is useful when one would like to approximate the channels with those that are quantized. This then leads to an easier way to test observability via a rank condition, e.g. when $\mathcal{X}$ is finite; see Section \ref{finiteEx}.
\end{remark}

\begin{remark}\label{Rm4} One should note that the definition is not one of invertibility; it only requires that there exists some $g$ and $N$ such that the error between the conditional expectation of $g(y_{[1,N]})$ given $X_1=x$, and $f(x)$ is small. In particular, $X_1$ is not necessarily, even approximately, recoverable given the measurements. Invertibility, however, would be a special case being a sufficient condition, as we will see in the examples. 

\end{remark}

\sy{
\begin{remark}\label{Rm55}[Recovery of Initial Probability Measure]
By our definition of observability, for every $f \in C_b(\mathbb{X})$, the value is $\langle \mu, f \rangle:= E_{\mu}[f(X)]$ determinable with arbitrary precision by the measurements (since $f$ is recovered, uniformly over a given compact set, with arbitrary precision). Since a countable collection of continuous and bounded functions uniquely distinguish probability measures \cite[Theorem 3.4.5]{ethier2009markov} (that is, such continuous and bounded functions form a {\it separating class}, see also \cite[p. 13]{Bil99}), this amounts to the recovery of the initial probability measure as more and more measurements are collected. This then leads to the conclusion that our definition implies Van Handel's definition given in (\ref{vanHandelDefn}), noted further below (note that this also applies for non-compact setups under Definition \ref{local_observabilityRelaxed} as every individual probability measure is tight). 
\end{remark}
}

\subsection{Filter stability under the observability definition}

The presented observability definition leads to predictor stability in the following sense.

\begin{theorem}\label{cont_bounded}
Let
\begin{equation}\label{absContMeas}
P^{\mu}|_{\mathcal{F}^{\mathcal{Y}}_{0,\infty} } \ll P^{\nu}|_{\mathcal{F}^{\mathcal{Y}}_{0,\infty} }. 
\end{equation}
If the POMP is observable, then $\pi_{n-}^{\mu}$ and $\pi_{n-}^{\nu}$ merge weakly as $n \to \infty$, $P^{\mu}$ a.s.
\end{theorem}

A sufficient condition for (\ref{absContMeas}) is that the priors satisfy $\mu \ll \nu$. The following assumption will allow us to use the recent results in \cite{KSYWeakFellerSysCont} (see also \cite{FeKaZg14}) and conclude the weak merging of the filter in expectation from the almost sure convergence of the predictor.

\begin{assumption}\label{Ali_assumption}
The measurement channel is continuous in total variation. That is for any sequence $a_n$ with $\lim_{n \to \infty} a_{n}=a \in \mathcal{X}$ we have $\sy{ \|G(\cdot| X_0=a_n)  -  G(\cdot| X_0=a) \|_{TV} \to 0}$.
\end{assumption}
An example of a channel which is continuous in total variation is as follows \cite[Section 2.2]{KSYWeakFellerSysCont}: $Y_n = F(X_n) + W_n$ where $F$ is continuous and $W_n$ admits a continuous density function (such as a Gaussian), where an analysis based on convolution and Scheff\'e's lemma leads to the conclusion. 

 \begin{theorem}\label{Ali_Theorem}
Let Assumption \ref{Ali_assumption} hold, if the predictor merges weakly $P^{\mu}$ a.s., then the filter merges weakly in expectation.
\end{theorem}



\subsubsection{Localized Observability for Non-Compact Signal Spaces}\label{relax_sec}

While the definition of observability that we introduced is valid for both compact and non-compact state spaces, it may be difficult to satisfy the definition in a non-compact state space with a uniform bound on the approximation error. This will be relaxed in the following, where we assume $ \mathcal{X}$ to be Euclidean.


\begin{definition}\label{local_predictable}
Given a compact set $K$, a POMP is called $K$ locally predictable if there exists a sequence of $\mathcal{F}^{\mathcal{Y}}_{0,n-1}$ (with \sy{$n \in \mathbb{N}$}) measurable mappings (random variables) $a_{n}:\mathcal{Y}^{n}\to \mathcal{X}$ such that
\begin{align*}
\pi_{n-}^{\nu}(K+a_{n})=1~P^{\mu}~a.s.
\end{align*}
for every $\mu \ll \nu$.
\end{definition}

This definition can be interpreted as follows. Regardless of the prior $\nu$, upon seeing observations $Y_{[0,n-1]}$ we can be sure $X_{n}$ lives in a compact set $K_{n}=K+a_{n}$. We can think of $a_{n}$ as a ``centring'' value based on the observations $Y_{[0,n-1]}$ and $K$ as the compact set around this centring value in which $X_{n}$ must live conditioned on observations $Y_{[0,n-1]}$. This is then paired with a definition of local observability.

\begin{definition}\label{local_observability}
Given a compact set $K$, a POMP is called $K$ locally observable if for every continuous and bounded function $f$, every sequence of numbers $a_{n}$, and every $\epsilon>0$, there exists a sequence of uniformly bounded measurable functions $g_{n}$ such that
\begin{align*}
\sup_{x \in K+a_{n}} \left|f(x)- \int_{\mathcal{Y}}g_{n}(y)G(dy|x) \right|\leq \epsilon
\end{align*}
for every $n \in \mathbb{N}$.
\end{definition}

 \begin{theorem}\label{local_thm}
Assume $\mu \ll \nu$ and that there exists a compact set $K$ such that the POMP is $K$ locally predictable and $K$ locally observable. Then the predictor merges weakly $P^{\mu}$ a.s..
\end{theorem}
The result above is intuitive as we can specify the compact set $K$, and the shifted sets $K_{n}=K+a_{n}$, over which we must approximate the function. However, the definitions can also be constructed taking a more relaxed approach and appealing to tightness rather than a probability one statement, but in this case it is more difficult to satisfy the local definition of observability. 

\begin{definition}\label{local_predictableRelaxed}
A POMP is called {\it locally predictable} if there exists a sequence of $\mathcal{F}^{\mathcal{Y}}_{0,n-1}$ (with \sy{$n \in \mathbb{N}$}) measurable mappings $a_{n}:\mathcal{Y}^{n}\to \mathcal{X}$ such that the family of measures
\begin{align*}
\tilde{\pi}^{\nu}_{n-}(\cdot) := \pi_{n-}^{\nu}(\cdot+a_{n})
\end{align*}
for every $\mu \ll \nu$, is a uniformly tight family of measures.
\end{definition}


\begin{definition}\label{local_observabilityRelaxed}

A POMP is called locally observable if for every continuous and bounded function $f$, every compact set $K$, every sequence of numbers $a_{n}$, and every $\epsilon>0$, there exists a sequence of uniformly bounded measurable functions $g_{n}$ such that
\begin{align*}
\sup_{x \in K+a_{n}} \left|f(x)- \int_{\mathcal{Y}}g_{n}(y)G(dy|x) \right|\leq \epsilon,
\end{align*}
\begin{align*}
\sup_{x \not\in K+a_{n}} \left| \int_{\mathcal{Y}}g_{n}(y)G(dy|x) \right|\leq 2\|f\|_{\infty}
\end{align*}
for every $n \in \mathbb{N}$.
\end{definition}

\begin{theorem}\label{local_thmRelaxed}
Assume $\mu \ll \nu$ and that the POMP is locally predictable and locally observable. Then the predictor merges weakly $P^{\mu}$ a.s.
\end{theorem}
An example is in Section \ref{nonCompactObservability}.

\subsection{Relations between various criteria for filter stability}\label{extend_sec}


It follows from Pinsker's inequality that relative entropy merging implies total variation merging, which in turn implies weak merging (by Definitions \ref{weak_def} and \ref{tv_def}). In this section we are interested conditions for when the converse direction holds, i.e. weak merging implies total variation or relative entropy merging. Recall the definition that for the measurement channel $G(Y_n \in \cdot | X_n=x)$ to be {\it dominated} in the sense that there exists a reference measure $\lambda$ such that $\forall x \in \mathcal{X}, G(dy | x_{n}=x) \ll \lambda$. Then, we define the Radon-Nikodym derivative
\[g(x,y):=\frac{dG(y_{n} \in \cdot|x_{n}=x)}{d\lambda (\cdot)}(y) \]
which serves as a likelihood function (a conditional probability density function). 


\begin{assumption}\label{lebesgue_cont_control2}
\begin{itemize}
\item[(i) ]$T(\cdot|x)$ is absolutely continuous with respect to a dominating measure $\phi$ for every $x \in \mathcal{X}$, so that $ t(x_1,x) =\frac{dT(\cdot|x)}{d\phi}(x_{1})$ where $t$ is continuous in $x$ for every $x_1 \in \mathcal{X}$.
\item[(ii)] $g(x,y)$ is bounded and continuous in $x$ for every fixed $y$. Furthermore, $g(x,y) > 0$ for all $x \in \mathcal{X}, y \in \mathcal{Y}$.
\end{itemize}
\end{assumption}

\begin{assumption}\label{lebesgue_cont}
$T(\cdot|x)$ is absolutely continuous with respect to a dominating measure $\phi$ for every $x \in \mathcal{X}$, so that $ t(x_1,x) =\frac{dT(\cdot|x)}{d\phi}(x_{1})$. The family of (conditional densities) $\{t(\cdot, x)\}_{x \in \mathcal{X}}$ is uniformly bounded and equicontinuous. 
\end{assumption}




\begin{theorem}\label{weak_to_total}
Let $\mu \ll \nu$. Let either one of the following hold:
\begin{itemize}
\item[i)] Assumption \ref{lebesgue_cont_control2}, or
\item[ii)] Assumption \ref{lebesgue_cont}. 
\end{itemize}
Then, if the predictor is stable in the weak sense $P^{\mu}$ a.s. then it is also stable in total variation $P^{\mu}$ a.s..
\end{theorem}
Since the total variation of any two probability measures is uniformly bounded, stability in the almost sure sense implies that in expectation (with the same also holding for predictors). Thus, Theorem \ref{weak_to_total} also presents condition for predictor merging in total variation in expectation. 
%
%

\begin{theorem}\label{tv_equal_placeholder}
The filter merges in total variation in expectation if and only if the predictor merges in total variation in expectation.
\end{theorem}

Recently, the filter stability results under total variation in expectation have been shown to be consequential in showing the optimality of finite memory control policies in Partially Observed Markov Decision Processes (see \cite[Section 4.3 and Theorem 9]{kara2020near} and \cite[Theorems 3.2, 3.3 and 4.1]{kara2021convergence}).

\begin{theorem}\label{re_tv_equal}
Assume there exists some finite $n$ such that $E^{\mu}[D(\pi_{n}^{\mu}\|\pi_{n}^{\nu})]<\infty$ and some $m$ such that $D(P^{\mu}|_{\mathcal{F}^{\mathcal{Y}}_{0,m}}\|P^{\nu}|_{\mathcal{F}^{\mathcal{Y}}_{0,m}})<\infty$. Then the filter is stable in relative entropy if and only if it is stable in total variation in expectation.
\end{theorem}
We note that both of the conditions on the finiteness of relative entropies in Theorem \ref{re_tv_equal} are minor and hold for example if $D(\mu\|\nu)<\infty$. In the special setup where the measurement sigma field is the trivial one (with no information), or more generally $Y_{n}$ is independent of $X_n$; the above recovers the following result which generalizes Barron \cite{BarronSorrento,BarronInfoMartingales} and Fritz \cite{Fritz73}, who had established the relative entropy convergence of sequences of probability measures for each time stage to the invariant measure for reversible Markov chains. This result also generalizes Theorem 5 of Harremo\"es and Holst \cite{HarremosHolst} which considers countable state space chains with a uniform irreducibility assumption. 

\begin{theorem}\label{setwiseConv}
Let $X_t$ be a Markov chain with $\pi$ being its unique invariant probability measure. Let $\pi_t$ denote the measure $P^{\pi_0}(X_t \in \cdot)$, where $X_0 \sim \pi_0$. Let $\pi_t \to \pi$ in total variation. If $D(\pi_{t_0}||\pi) < \infty$ for some $t_0 < \infty$, then \[D(\pi_t || \pi) \downarrow 0.\]
In particular, for an aperiodic positive Harris recurrent Markov chain, if $D(\pi_{t_0}||\pi) < \infty$ for some $t_0 < \infty$, then $D(\pi_t || \pi) \downarrow 0$.
\end{theorem}

\begin{proof} The proof follows directly from Theorem \ref{re_tv_equal}. In the special case of positive Harris recurrence, the result follows since for aperiodic positive Harris recurrent Markov chains, $\pi_t \to \pi$ in total variation (see Theorem 13.0.1 in \cite{MeynBook}). \end{proof}

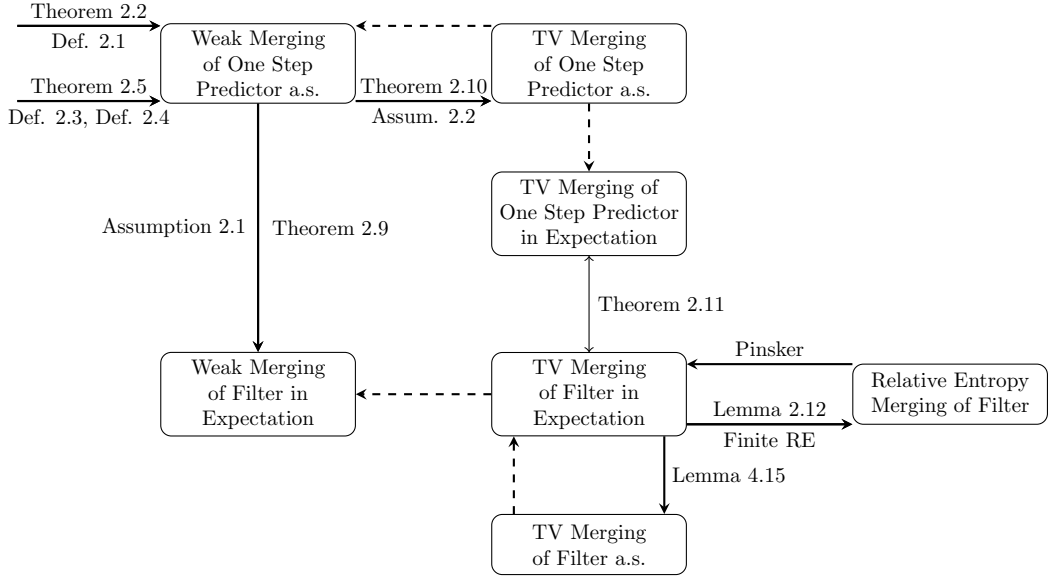
\begin{figure}[t]
\begin{center}
\begin{tikzpicture}[scale=0.4, every node/.style={transform shape}]

\node (weak_pred) [terminal,  text width = 3cm] {Weak Merging of One Step Predictor a.s.};

\node (tv_pred_a_s) [terminal, right of = weak_pred, xshift = 5.5cm, text width = 4cm] {TV Merging of One Step Predictor a.s.};

\node (tv_pred_exp) [terminal, below of = tv_pred_a_s, yshift = -1.5 cm, text width = 4cm] {TV Merging of One Step Predictor in Expectation};

\node (tv_filt) [terminal, below of = tv_pred_exp, yshift = -2  cm, text width = 4cm] {TV Merging of Filter in Expectation};

\node (re_filt) [terminal, right of = tv_filt, xshift = 7  cm, text width = 4cm]  {Relative Entropy Merging of Filter};

\node (tv_filt_as) [terminal, below of = tv_filt, yshift = -2.5 cm, text width = 4 cm] {TV Merging of Filter a.s.};

\node (weak_filt) [terminal, left of = tv_filt, xshift = -5.5 cm, text width = 3cm] {Weak Merging of Filter in Expectation};

Arrows
\draw[arrow, transform canvas={yshift=0.5cm}] +(-4cm, 0cm) -- node[anchor=south] {Theorem \ref{cont_bounded}} (weak_pred);

\draw[arrow, transform canvas={yshift=0.5cm}] +(-4cm, 0cm) -- node[anchor=north] {Def. \ref{observable}} (weak_pred);

\draw[arrow, transform canvas={yshift=-0.5cm}] +(-4cm, 0cm) -- node[anchor=south] {Theorem \ref{local_thm}} (weak_pred);

\draw[arrow, transform canvas={yshift=-0.5cm}] +(-4cm, 0cm) -- node[anchor=north] {Def. \ref{local_predictable}, Def. \ref{local_observability}} (weak_pred);


\draw[arrow, transform canvas={yshift=-0.5cm}] (weak_pred) -- node[anchor=south] {Theorem \ref{weak_to_total}} (tv_pred_a_s);

\draw[arrow, dashed] (tv_pred_a_s) -- node[anchor=south] {} (tv_pred_exp);


\draw[arrow, transform canvas={yshift=-0.5cm}] (weak_pred) -- node[anchor=north] {Assum. \ref{lebesgue_cont}} (tv_pred_a_s);

\draw[arrow, transform canvas={yshift=+0.5cm},dashed] (tv_pred_a_s) -- node[anchor=north] {}(weak_pred);

\draw[arrow, transform canvas={yshift=-0.4cm}] (tv_filt) -- node[anchor=south] {Theorem \ref{re_tv_equal}} (re_filt);
\draw[arrow, transform canvas={yshift=-0.4cm}] (tv_filt) -- node[anchor=north] {Finite RE} (re_filt);

\draw[arrow, transform canvas={yshift=+0.4cm}] (re_filt) -- node[anchor=south] {Pinsker} (tv_filt);

\draw[arrow, transform canvas={xshift=+1cm}] (tv_filt) -- node[anchor=west] {Lemma \ref{almost_sure_stability}} (tv_filt_as);

\draw[arrow, transform canvas={xshift=-1cm},dashed] (tv_filt_as) -- node[anchor=south] {} (tv_filt);

\draw[<->] (tv_pred_exp) -- node[anchor=center, xshift = 1.2cm] {Theorem \ref{tv_equal_placeholder}} (tv_filt);

\draw[arrow] (weak_pred) -- node[anchor=center, xshift=1.2cm] {Theorem \ref{Ali_Theorem}} (weak_filt);

\draw[arrow] (weak_pred) -- node[anchor=center, xshift=-1.4cm] {Assumption \ref{Ali_assumption}} (weak_filt);

\draw[arrow, dashed] (tv_filt) -- node[anchor=center] {} (weak_filt);

\end{tikzpicture}
\caption{Flow of Ideas and Conditions for Filter Stability}
\label{fig:proof_diagram}
\end{center}
\end{figure}

\subsection{Literature review and comparison of results}\label{lit_review}

For deterministic linear systems, exact recovery of any initial condition with measurements available until some finite time is defined as observability and is characterized by an observability rank condition in both continuous and discrete-time \cite{Chen}. For linear systems, such an observability definition is global (as it applies for all initial states) and is universal in the control policies applied, as the control policy does not affect the estimation errors (known as the {\it no-dual effect} \cite{bartse74} property). For non-linear systems, however, due to the challenges in the analysis which prevent globality as well as control-dependence, more modest and localized definitions are to be imposed: For deterministic continuous-time non-linear systems \cite{hermann1977nonlinear} and \cite{sontag1984concept} present local indistinguishability conditions with subtle differences, and establish relations with Lie-theoretic characterizations which generalize observability rank conditions for non-linear systems defined locally. For discrete-time deterministic models, observability has also been defined by invertibility or exact recovery of an initial state, locally, given measurements with finitely many observations. Nijmeier \cite{nijmeijer1982observability} developed discrete-time analogues of the observability notions presented in \cite{hermann1977nonlinear} (see also \cite{sontag1984concept} for sampled continuous-time systems). Liu and Bitmead \cite{liu2011stochastic,liu2010observability} introduce a non-linear stochastic observability definition through entropy, where the conditional entropy of the hidden state given measurements not being the same as the unconditional entropy implies observability. Ugrinoovski \cite{ugrinovskii2003observability} also presents an information theoretic formulation, and defines observability as an informativeness condition.

In the filtering literature for control systems, the classical setup involves the linear Gaussian system. The filter in this case is the celebrated Kalman filter, where the finite-dimensional Kalman filter is computed recursively using the Riccati equation. Under linear observability and controllability conditions, the Riccati equation admits a unique solution \cite{KushnerKalmanFilter,kushner2014partial,Caines}, which is the unique limit of the Riccati recursions regardless of the initialization. Thus, the Kalman Filter is stable with respect to incorrect, though still Gaussian, priors under the aforementioned conditions (we note that partial convergence and robustness results on the asymptotic equivalence of conditional expectations and linear estimates for non-Gaussian priors for linear systems are reported in \cite{sowers1992discrete}). \sy{The time-varying linear system setup has been studied in \cite{anderson1981detectability}.}

In a recent paper, we studied the implications on filter stability in robust control \cite{MYRobustControlledFS}. Implications on finite memory approximations in optimal stochastic control have been presented in \cite{kara2020near,kara2021convergence}. It is worth pointing out that there has been a recurrent theme on the the duality between controllability and observability, for a recent work in this direction see \cite{kim2022duality,kim2019lagrangian}. Filter stability for deterministic systems under noisy measurements has recently been studied in \cite{reddy2019stability}.

A strict version of our observability definition is captured in \cite[Equation 1.7]{chigansky2006role}. The idea there is to express, exactly, a continuous function $f(x)$ by integrating a measurable function $g(y)$ over the conditional distribution for $Y$ given $X=x$. 
A fundamental result which pairs with observability is that of Blackwell and Dubins \cite{blackwell1962merging}, an implication of which \cite{chigansky2006role} independently arrived at. Blackwell and Dubins use martingale convergence theorem to show that if $P$ and $Q$ are two measures on a fully observed stochastic process $\{X_{n}\}_{n=0}^{\infty}$ with $P \ll Q$, then the conditional distributions on the future based on the past merge in total variation $P$ a.s., that is $P$ a.s.
\begin{align*}
\|P(X_{[n+1,\infty)} \in \cdot|X_{[0,n]})-Q(X_{[n+1,\infty)} \in \cdot|X_{[0,n]})\|_{TV} \to 0.
\end{align*} 
\cite{van2009observability} introduces a definition of observability for POMPs. Namely, a system is observable if every prior results in a unique probability measure on the measurement sequences;
 \begin{align}\label{vanHandelDefn}
 P^{\mu}|_{\mathcal{F}^{\mathcal{Y}}_{0,\infty}}=P^{\nu}|_{\mathcal{F}^{\mathcal{Y}}_{0,\infty}} \implies \mu = \nu.
 \end{align}
\cite{van2009observability} shows that the above leads to filter stability for continuous-time models with compact state space. \cite{van2009uniform} extends these results to non-compact state spaces, where {\it uniform observability} is introduced. 
The result of Blackwell and Dubins \cite{blackwell1962merging} is utilized to show that uniform observability would imply filter stability in bounded Lipschitz distance \cite{van2008discrete}. Nonetheless, this condition is implicit; \cite{van2008discrete} only studies the measurement channel where $h(x,z)=f(x)+z$ where $f^{-1}$ is uniformly continuous and $Z$ must have an everywhere non-zero characteristic function (e.g. a Gaussian distribution). 
 %
For a compact state space, \cite{van2009uniform} has established that uniform observability and observability are equivalent notions. We also note that for a finite state space with a non-degenerate measurement channel (i.e. likelihood function $g(x,y) >0$), stability can be fully characterized via observability and a detectability condition \cite{van2009observability}, \cite[Theorem V.2]{van2010nonlinear} or \cite[Theorems 2.7 and 3.1]{chigansky2010complete}.

\sy{As noted in Remark \ref{Rm55}, our definition implies (\ref{vanHandelDefn}). (\ref{vanHandelDefn}) is a  statement of invertibility with no clear guidance on how to test this property; our definition is explicitly given in a test function formulation, making it more interpretable and easier to apply to various systems of interest. Additionally, in the work studied here, we consider discrete time processes and thus the predictor and the filter are distinct objects. Our definition of observability only implies the weak merging of the predictor almost surely, not the filter directly. Conditions are needed to relate the merging of the predictor to that of the filter. This is also addressed in our paper building also on recent results on the regularity properties of non-linear filters from \cite{KSYWeakFellerSysCont} (see also \cite{FeKaZg14}).} 

In early work by Kunita, \cite{kunita1971asymptotic}, the stability of the filter process is studied in light of the limit sigma fields of the processes (e.g. $\mathcal{F}^{\mathcal{Y}}_{0,\infty}, \mathcal{F}^{\mathcal{X}}_{0,\infty}$). Kunita's work unfortunately made a technical error on the exchange of orders in supremum and intersection operations on sigma fields: A concise derivation of the corrected result is presented in \cite[Equation 1.10]{chigansky2009intrinsic}: here, we are presented with a sufficient and necessary condition for the merging of the filter in total variation in expectation based on comparing the sigma fields $\mathcal{F}^{\mathcal{Y}}_{0,\infty}$ and $\bigcap_{n \geq 0}\mathcal{F}^{\mathcal{X}}_{n,\infty} \vee \mathcal{F}^{\mathcal{Y}}_{0,\infty}$.  That is the filter merges in total variation in expectation if and only if:
\begin{align}
E^{\nu}[\frac{d\mu}{d\nu}(X_{0})|\mathcal{F}^{\mathcal{Y}}_{0,\infty}]=E^{\nu}[\frac{d\mu}{d\nu}(X_{0})|\bigcap_{n \geq 0}\mathcal{F}^{\mathcal{X}}_{n,\infty} \vee \mathcal{F}^{\mathcal{Y}}_{0,\infty}]~~~P^{\mu}~a.s.\nonumber
\end{align}

%

 
Relative entropy as a measure of discrepancy between the true filter and the incorrectly initialized filter is studied by Clark, Ocone, and Coumarbatch in \cite{clark1999relative}. Here, the authors considered the filtering problem in continuous time and with a dominated measurement channel. The authors established the relative entropy of the true filter and the incorrect filter as a supermartingale, and its convergence to a limit. However, the paper did not establish the convergence to zero. A notable setup where actual convergence (of the relative entropy) to zero is established is the (rather specific) Bene{\v{s}} filter studied in \cite{ocone1999asymptotic}. This problem also has relations to the relative entropy convergence of Markov chains to invariance: in the case where the measurements are trivial, the convergence problem reduces to what has been studied in \cite{BarronSorrento, BarronInfoMartingales,Fritz73,HarremosHolst,harremoes2010thinning} on relative entropy convergence of Markov chains to invariant measures.

{\bf Contributions and comparison with the literature.} In view of the review above, our contributions are as follows:
\begin{itemize}[leftmargin=*]
\item[i)] [Stochastic observability] In Section \ref{obs_sec}, we present a definition of stochastic observability. This definition is functionally explicit and testable, and due to its functional approximation characterization, it allows various analytical methods to be applicable for verification (see Remarks \ref{Rm1} through Remark \ref{Rm4}). 

Under this definition, we establish predictor stability (in the weak convergence/merging sense). We note that observability, for the discrete time case as studied here, only implies weak merging of the predictor almost surely, not the filter directly.  This is addressed in our paper building also on recent results on the regularity properties of non-linear filters from \cite{KSYWeakFellerSysCont} (see also \cite{FeKaZg14}). 

 We also note that the Blackwell and Dubins theory of merging on which our approach builds (similar to \cite{van2009observability,van2009uniform}), applies for infinite sequences of future events (i.e. $P^{\mu}(Y_{[n+1,\infty)}|Y_{[0,n]})$), this is utilized in our definition of  $N$ step observability leading to application examples of broad generality. Accordingly, our definition is not only a function of the measurement channel, but also of the system dynamics; unlike some related results in the literature. 
 
 Additionally, we provide several examples in Section \ref{examples}. 

\item[ii)] [On various convergence and merging criteria] We establish new results relating various criteria for filter stability (as depicted in Figure \ref{fig:proof_diagram}), independent of the mechanism used to arrive at filter stability: we study filter stability under weak merging and total variation merging in expectation and almost surely, as well as relative entropy. In Section \ref{extend_sec}, (a) We place mild assumptions on the transition/measurement kernels to extend weak merging of the predictor to total variation merging. (b) We show that total variation merging of the predictor and filter are equivalent, and (c) under a mild finiteness condition on the relative entropy sequence, we also establish equivalence of relative entropy merging and total variation merging. 

Using the chain rule for relative entropy, the relative entropy error was shown to be a non-increasing sequence by Clark et.al. \cite{clark1999relative}, but {\it its convergence to zero was not established}, except for the specific case of the Bene{\v{s}} filter in \cite{ocone1999asymptotic}. Theorem \ref{re_tv_equal} establishes the equivalence between relative entropy merging and total variation, and thus convergence of the relative entropy error to zero is proven here (we note that this is a result which is hinted at in the literature, see \cite[Remark 4.2]{chigansky2009intrinsic} or \cite[Remark 5.9]{van2009stability} but not explicitly proven). This result applies to setups beyond filter stability: In the case where the measurements are trivial, Theorem \ref{setwiseConv} generalizes \cite{BarronSorrento, BarronInfoMartingales,Fritz73,HarremosHolst} on relative entropy convergence of Markov chains to invariant measures, where the first references due to Barron and Fritz had considered reversible Markov chains and the latter due to Harremo\"es and Holst focused on countable state Markov chains under a uniform irreducibility assumption. On Theorem \ref{tv_equal_placeholder}, we note first that much of the literature focuses on continuous time, where the predictor is not used in the analysis. In discrete time, \cite[Lemma 4.2]{van2008discrete} proves that the merging of the predictor in total variation in expectation implies that of the filter. However this result relies on a domination assumption in the measurement channel and the specific structure of the filter recursion equation \cite[Equation 1.4]{chigansky2009intrinsic}. Theorem \ref{tv_equal_placeholder} is, accordingly, a more general result.

\item[iii)] [Implications to near optimality of finite window policies in POMDPs] Our findings lead to practically relevant and mathematically consequential implications to robustness and approximations for controlled partially observable models; i.e., POMDPs:  \cite{MYRobustControlledFS} has studied controlled filter stability where it was shown that one-step observability introduced here leads to stochastic observability universal over all admissible control policies, which then leads to refined robustness results when compared with \cite{KYSICONPrior}. In this paper, we consider the control-free case, which allows us to consider $N$-step observability, with $N > 1$. Additionally, we present numerous explicit examples, which, in the one-step observable setup, is then directly applicable to such robustness results. Recently, filter stability results under total variation (as well as weak convergence under slightly more restrictive setups) have been shown to be consequential in showing the optimality of finite memory control policies in Partially Observed Markov Decision Processes (POMDPs); see \cite[Section 4.3 and Theorem 9]{kara2020near} and \cite[Theorems 3.2, 3.3 and 4.1]{kara2021convergence} where connections with weak merging and total variation merging are made explicit in the approximation error bounds (see \cite{white1994finite} for an earlier study where the dependence on filter stability is implicit; further related recent studies include \cite{golowich2022planning}). Accordingly, the results in this paper, notably Theorems \ref{weak_to_total} and \ref{tv_equal_placeholder}, are directly applicable in showing that with merging under total variation in expectation, one can show that optimal policies for POMDPs can be approximated by those which use only finite window of measurements and control actions.
\end{itemize}

\section{Observable System and Measurement Channel Examples}\label{examples} We note that in this section, it will be more convenient to describe our measurement channels via the equivalent functional realization (see (\ref{dynamicsEqn})), with explicit noise variable $Z_{n}$ and a measurement function $Y_{n}=h(X_{n}, Z_{n})$, and thus this will the convention we will use to define the measurement channel $G$ for the examples presented in the following.

\subsection{Finite state and noise space}\label{finiteEx}
Consider a finite setup $\mathcal{X}=\{a_{1},\cdots, a_{n}\}$ and $\mathcal{Z}=\{b_{1},\cdots,b_{m}\}$. Now, assume $h(x,z)$ has $K$ distinct outputs, where $1\leq K\leq (n)(m)$ and $\mathcal{Y}=\{c_{1},\cdots,c_{K}\}$. \sy{We note that for such a setup, there is already a sufficient and necessary condition provided in \cite[Theorem V.2]{van2010nonlinear}. However, we examine this case to show that our definition is equivalent to the sufficient direction of this theorem, which is the notion of observability presented in\cite{van2009observability}.}

For each $x$, \sy{$h_{x}(\cdot):=h(x, \cdot)$} can be viewed as a partition of $\mathcal{Z}$, assigning each $b_{i} \in \mathcal{Z}$ to an output level $c_{j}\in \mathcal{Y}$. We can track this by the matrix $H_{x}(i,j)=1$ if $h_{x}(b_{i})=\sy{c_{j}}$ and zero else. Let $Q$ be the $1 \times m$ vector representing the probability measure of the noise. Let us first consider the one step observability. Let $g(c_{i})=\alpha_{i}$, with $\alpha = \begin{bmatrix}\alpha_1 \\ \alpha_2 \\ \vdots \\ \alpha_K \end{bmatrix}$, and $\int_{\mathcal{Z}} g(h(x,z))Q(dz)$ $=:QH_{x}\alpha
$. Therefore, any function $f(x)$ can be expressed as a $n \times 1$ vector and hence the question reduces to finding a vector $\alpha$ so that $f =  QH \alpha$, and the system is one step observable if and only if the matrix
$
A\equiv\begin{pmatrix}
QH_{a_{1}}\\
\vdots\\
QH_{a_{n}}
\end{pmatrix}
$
is rank $n$. 
%

Consider then $N$ step observability. We wish to solve equations of the form
\begin{align}
f(x)=\int_{\mathcal{Y}^{N}}g(y_{[1,N]})dP^{\mu}(y_{[1,N]}|x_{1}=x)\label{finite_N_obs}
\end{align}
With knowledge of $Q,h(\cdot,\cdot)$ and $T$ we can directly compute the transition kernel for the joint measure $Y_{[1,n]}|X_{1}$, however the size of this matrix is $n$ by $K^{n}$ where $|\mathcal{X}|=n,|\mathcal{Y}|=K$ so complexity grows exponentially. We can deduce a sufficient, but not necessary, condition for $n$ step observability using the marginal conditional measures. Consider that
$P^{\mu}(y_{k}\in \cdot|X_{1}=a_{j})=T(a_{j}|:)T^{k-2}A,~k\geq 2$
where $T(a_{j}|:)$ represents the $j^{th}$ row of the transition matrix. Note that these are all $1 \times K$ vectors and represent the marginal measures of $Y_{k}|X_{1}$. 
Consider the class of functions $\mathcal{G}^{n}=\{g:\mathcal{Y}^{n} \to \mathbb{R}\}$ and a subclass $\mathcal{G}^{n}_{LC}=\{g(y_{[1,n]})=\sum_{i=1}^{n}\alpha_{i}g_{i}(y_{i})|\alpha_{i} \in \mathbb{R},g_{i}\in \mathcal{G}^{1}\}$. That is, a linear combination of functions of the individual $y_{i}$ values. We can use these functions to deduce a sufficient, but not necessary, condition for observability.

\begin{lemma}\label{marginal_rank}
Assume that $|\mathcal{X}|=n$ and define the matrix
\begin{align*}
M=\begin{pmatrix}
A&
TA&
\cdots&
T^{n-1}A
\end{pmatrix}
\end{align*}
which is $n \times nK$ where $K=|\mathcal{Y}|$. If $M$ is rank $n$, then the system is $n$ step observable. Furthermore, if $M$ is not rank $n$, appending more blocks of the form $T^{k}A$ for $k\geq n$ will not increase the rank of $M$.
\end{lemma}
\begin{proof}
Begin with (\ref{finite_N_obs}), consider a restriction to $\mathcal{G}^{n}_{LC}$, that is we require $g$ to be of the form $g(y_{[1,n]})=\sum_{i=1}^{n}g_{i}(y_{i})$. Denote the $(nK) \times 1$ vector\\
$\sy{\alpha=(g_{1}(c_{1}),\cdots, g_{1}(c_{K}),\cdots, g_{n}(c_{1}),\cdots,g_{n}(c_K))}$. Then 
\begin{align*}
&f(x)=\sum_{i=1}^{n}P^{\mu}(y_{i}\in \cdot|X_{1}=x)\begin{pmatrix}
g_{i}(c_{1})\\
\vdots\\
g_{i}(c_{K})
\end{pmatrix} \\
&=\begin{pmatrix}
QH_{x}&T(x|:)A&\cdots&T(x|:)T^{n-2}A
\end{pmatrix}
\alpha
\end{align*}
We can then see that this matrix is the $j$ row of $M$ when $x=a_{j}$, therefore we have
$
\begin{pmatrix}
f(a_{1})\\
\vdots\\
f(a_{n})
\end{pmatrix}=\begin{pmatrix}
A&
TA&
\cdots&
T^{n-1}A
\end{pmatrix}
\alpha
$. If $M$ is rank $n$, then any function $f:\mathcal{X} \to \mathbb{R}$ can be expressed as a vector $g$ put through matrix $M$ and the system is observable.

Consider if $M$ is not rank $n$ and if we append another block $T^{n}A$  to $M$. By the Cayley-Hamilton theorem, $T^{n}$ is a linear combination of lower powers of $T$, e.g. $T^{n}=\sum_{i=0}^{n}\alpha_{i}T^{i}$ for some coefficients $\alpha_{i}$. Therefore this additional block is a linear combination of the previous blocks, and adds no dimension to the matrix $M$.
\end{proof}

If the conditions of this lemma fail, i.e. $M$ is not rank $n$, that means integrating $g$ over the marginal measures cannot generate any $f$ function. Yet the product of the marginal measures is not the joint measure since the $Y_{i}|X_{1}$ are not independent. Hence, working with the marginal measures only is not enough to determine observability as also noted in \cite[Remark 13]{van2009observability} in a slightly different setup. 

Consider the following example.
Let $\mathcal{X}=\{1,2,3,4\}$ and $Y=1_{X\leq 2}$. This can be realized as
\begin{align*}
A=\begin{pmatrix}
QH_{1}\\
\vdots\\
QH_{4}
\end{pmatrix}&=\begin{pmatrix}
0&1\\
0&1\\
1&0\\
1&0
\end{pmatrix}
\end{align*}
Consider the following transition kernel,
\begin{align*}
T=
\begin{pmatrix}
0&\frac{1}{4}&\frac{1}{4}&\frac{1}{2}\\
\frac{1}{2}&0&0&\frac{1}{2}\\
0&\frac{1}{4}&\frac{1}{4}&\frac{1}{2}\\
\frac{1}{2}&0&0&\frac{1}{2}
\end{pmatrix}
\end{align*}
Notice that the odd and even rows are identical. If we consider the marginal measures of $Y_{1}|X_{1},\cdots,Y_{4}|X_{1}$ we have the matrix
\begin{align*}
&\begin{pmatrix}
A&
\cdots&
T^{3}A
\end{pmatrix}=\\
&
\begin{pmatrix}
0	&1	&0.75	&0.25	&0.5625	&0.4375	&0.609375	&0.390625\\
0	&1	&0.50	&0.50	&0.6250	&0.3750	&0.593750	&0.406250\\
1	&0	&0.75	&0.25	&0.5625	&0.4375 &0.609375	&0.390625\\
1	&0	&0.50	&0.50	&0.6250 &0.3750	&0.593750	&0.406250
\end{pmatrix}
\end{align*}
which is only rank 3, not rank 4. Therefore, we cannot use the marginal measures to determine observability.

 However, if we consider the joint measure of $(Y_{1},Y_{2})|X_{1}$ we have the matrix
\begin{align*}
A'=\begin{pmatrix}
0&0&\frac{3}{4}&\frac{1}{4}\\
0&0&\frac{1}{2}&\frac{1}{2}\\
\frac{3}{4}&\frac{1}{4}&0&0\\
\frac{1}{2}&\frac{1}{2}&0&0
\end{pmatrix}
\end{align*}
Where row $i$ is conditioned on $x=i$ and the columns are ordered in binary $y_{2}y_{1}$, e.g. $P(y_{1}=1,y_{2}=0|x_{1}=2)$ is row 2 column 3. This matrix is full rank, hence the system is $N$ step observable with $N=2$, even though the marginal measures failed to be full rank.

\subsection{Compact state and noise spaces with affine observations}\label{example1}
Consider $\mathcal{X},\mathcal{Z}$ as compact subsets of $\mathbb{R}$ and let $h(x,z)=a(z)x+b(z)$ for some functions $a,b$ where the image of $\mathcal{Z}$ under $a$ and $b$ is compact (this ensures that $\mathcal{Y}$ is compact). Note that for a fixed choice of $z$, this is an affine function of $x$. We will arrive at sufficient conditions for one step observability. Since $\mathcal{X}$ is compact, the set of polynomials is dense in the set of continuous and bounded functions. Therefore, rather than working with a function $f \in C_{b}(\mathcal{X})$ without loss of generality we assume $f$ is a polynomial. Let $\mathcal{M}_{b}(\mathbb{R})$ represent the measurable and bounded functions on the real line and consider the mapping
\begin{align*}
&S:\mathcal{M}_{b}(\mathbb{R}) \to C_{b}(\mathbb{R})&S(g)(\cdot)\mapsto \int_{Z}g(h(\cdot,z))Q(dz)
\end{align*}
Let $\mathbb{R}[x]_{n}$ represent the polynomials on the real line up to degree $n$. Then we have that $S(g)$ is invariant on $\mathbb{R}[x]_{n}$, that is if $g$ is polynomial of degree $n$ then $S(g)$ is a polynomial of degree $n$. Furthermore, the coefficients of $S(g)(x)=\sum_{i=0}^{n}\beta_{i}x^{i}$ can be related to the coefficients of $g(x)=\sum_{i=0}^{n}\alpha_{i}x^{i}$ by a linear transformation. Define $
N(i,k)=\binom{i}{k}E(a(Z)^{k}b(Z)^{i-k})
$
then by recursive application of binomial theorem we have
\begin{align*}
\begin{pmatrix}
\beta_{0}\\
\beta_{1}\\
\beta_{2}\\
\vdots\\
\beta_{n}
\end{pmatrix}&=\begin{pmatrix}
N(0,0)&N(1,0)&\cdots&N(n,0)\\
0&N(1,1)&\cdots&N(n,1)\\
\vdots&\ddots&\ddots&\vdots\\
0&\cdots&0&N(n,n)
\end{pmatrix}\begin{pmatrix}
\alpha_{0}\\
\alpha_{1}\\
\alpha_{2}\\
\vdots\\
\alpha_{n}
\end{pmatrix}
\end{align*}
if we want to generate any polynomial, we require this matrix to be invertible, and since it is upper triangular this amounts to none of the diagonal entries being zero, that is $E[a(z)^{n}]\neq 0~\forall n \in \mathbb{N}$. Furthermore, we want  $g$ to be bounded so we must have $N(n,k)<\infty~\forall n \in \mathbb{N},k \in \{0,\cdots,i\}$.
\subsection*{Example}
Consider $\mathcal{X}=[-10,10]$, $\mathcal{Z}=[-1,1]$, $Z\sim \text{Uni}([-1,1])$ and $y=z^{2}x+z$. We then have $\mathcal{Y}=[-11,11]$. For any $n \in \mathbb{N}$ we have
\begin{align*}
E[a(z)^{n}]&=\frac{1}{2}\int_{-1}^{1}z^{2n}dz=\frac{1}{2n+1}\neq 0
\end{align*}
additionally, for any $n \in \mathbb{N},k \in \{0,\cdots,n\}$ we have
\begin{align}
N(n,k)&=\binom{n}{k}E(a(z)^{k}b(z)^{n-k})=\binom{n}{k}E(z^{n+k}) \nonumber  \\
&=\binom{n}{k}\frac{1}{n+k+1}<\infty \nonumber
\end{align}
\subsection{A non-linear measurement function}
Consider $\mathcal{X}$ as a compact subset of $\mathbb{R}$, $\mathcal{Z}=\mathbb{R}$. Let $h(x,z)=1_{x >z}x+1_{x\leq z}z$ and assume that $Q$ admits a density with respect to Lebesgue. We have
\begin{align*}
\int_{\mathcal{Z}}g(h(x,z))Q(dz)&=\int_{-\infty}^{x}g(x)q(z)dz+\int_{x}^{\infty}g(z)q(z)dz
\end{align*}
again, we can approximate any continuous and bounded function $f$ on $\mathcal{X}$ as polynomial, so we assume $f$ is differentiable. We have
\begin{align*}
f(x)&=\int_{-\infty}^{x}g(x)q(z)dz+\int_{x}^{\infty}g(z)q(z)dz\\
f'(x)&=g(x)q(x)+\int_{-\infty}^{x}g'(x)q(z)dz-g(x)q(x) \\
& =g'(x)Q(Z\leq x)
\end{align*}

Since $\mathcal{X}$ is compact there exists some $x_{\min} \in \mathbb{R}$ such that $x_{\min}<x$, $\forall x\in \mathcal{X}$. We require for some $\epsilon>0$ that $Q(Z<x_{min})>\epsilon$. This condition says every $x\in \mathcal{X}$ has some positive probability of being observed through $h(x,z)$ and we will not always get pure noise. Then we have
\begin{align*}
g'(x)&=1_{\mathcal{X}}(x)\frac{f'(x)}{Q(Z\leq x)}\\
g(x)&=c+\int_{-\infty}^{x}1_{\mathcal{X}}(u)\frac{f'(u)}{Q(Z\leq u)}du
\end{align*}
for some constant $c$. Therefore, we only need to define $g$ over $\mathcal{X}$. Furthermore, we require $g$ to be bounded, which is implied if $g'$ is bounded since $g$ is only defined over a compact space.



%
%
%


\subsection{Local observability for a non-compact state space}\label{nonCompactObservability} 

We now study a system which does not have a compact state signal space and satisfies the definitions of local predictability and local observability, so that we can apply Theorem \ref{local_thmRelaxed}. Consider the POMP with the following transition and measurement kernels
\begin{align*}
X_{n+1}&=X_{n}+N(1,1)\\
Y_{n}&=\begin{cases}
X_{n}+1&w.p.~\frac{1}{2}\\
X_{n}-1&w.p.~\frac{1}{2}
\end{cases}
\end{align*}
We will first show this system is locally predictable. Given an observation $Y_{n-1}$, it must be that $X_{n-1}=Y_{n-1}-1$ or $Y_{n-1}+1$, therefore any filter at time $n-1$ will consist of two point masses at $Y_{n-1}-1$ and $Y_{n-1}+1$ with the probability of these two points dependent on the prior. Therefore the predictor at time $n$ will be a convex combination of Gaussian random variables \sy{$\alpha_{n} {\cal N}(Y_{n-1}, 1)+(1-\alpha_{n}){\cal N}(Y_{n-1}+2,1)$} where $\alpha_{n}$ is determined by the prior. 

However, regardless of the value of $\alpha$  for any $\epsilon>0$ have some compact set $K_{\epsilon}$ such that $\pi_{n-}^{\nu}(K_{\epsilon}+Y_{n-1}) >1-\epsilon$ for any choice of $\nu$. Therefore the system is locally predictable.

For local observability, assume $K$ is an interval $[-M,M]$ for some whole number $M>0$ and pick a centering value $a$. Fix a continuous and bounded function $f$. We wish to demonstrate a function $g$ that approximates $f$ well over $K+a$ when integrated over the the measurement channel. $g$ must be bounded with a bound that does not depend on $a$.
If we define $g(y)$ recursively as follows:
\begin{align*}
g(y)&=\begin{cases}
0&y< -M+a+1\\
2f(y-1)&y \in [-M+a+1, -M+a+3)\\
2f(y-1)-g(y-2)&y  \in [-M+a+3, M+a+1]\\
-g(y-2)& y > M+a+1
\end{cases}
\end{align*}
$g$ is akin to a telescoping sum in that it cancels out its own previous values. We have
\begin{align*}
\int g(h(x,z))Q(dz)&=\frac{1}{2}(g(x+1)+g(x-1))
\end{align*}
For $x < -M+a$ we have $x-1<x+1<-M+a+1$ hence $g(x-1)=g(x+1)=0$. 
For $x \in [-M+a,-M+a+2)$ we have $g(x+1)=2f(x+1-1)=2f(x)$ while $g(x-1)=0$. Then for $x \in [-M+a+2, M+a]$ we have 
\begin{align*}
g(x+1)=2f(x+1-1)-g(x+1-2)=2f(x)-g(x-1)
\end{align*}
which will cancel with the other $g(x-1)$ term, hence the telescoping. For $x>M+a$ we have $g(x+1)=-g(x+1-2)=-g(x-1)$ hence it will cancel with the previous value.

In each iteration of telescoping, $\|g\|_{\infty}$ increases by at most $2\|f\|_{\infty}$, there are $2M$ iterations of telescoping so the overall bound on $\|g\|_{\infty}$ is $4M\|f\|_{\infty}$. Therefore we have
\begin{align*}
\|g\|_{\infty}&\leq 4M\|f\|_{\infty}\\
\int g(h(x,z)Q(dz)&=f(x) & x \in [-M+a,M+a]\\
\left|\int g(h(x,z)Q(dz)\right|&=0 &x \not\in [-M+a,M+a]
\end{align*}
this proves local observability.



\section{Proofs}\label{proof_sec}

\subsection{Observability: Proof of Theorem \ref{cont_bounded}}\label{proofcont_bounded}
\begin{lemma}\label{bayes_rearrange}
Let $g$ be a bounded and measurable function on $(\mathcal{Y}^{k+1},\mathcal{B}(\mathcal{Y}^{k+1}))$. For any initial prior $\mu$ we have
\begin{align}
&\int_{\mathcal{Y}^{k+1}}g(y_{[n,n+k]})P^{\mu}(dy_{[n,n+k]}|Y_{[0,n-1]}) \nonumber \\
&=\int_{\mathcal{X}}\int_{\mathcal{Y}^{k+1}}g(y_{[n,n+k]})P(dy_{[n,n+k]}|X_{n}=x_{n})\pi_{n-}^{\mu}(dx_{n})\label{bayes_RHS}
\end{align}
\end{lemma}
\smallskip
\begin{proof}
\begin{align*}
&\int_{\mathcal{Y}^{k+1}}g(y_{[n,n+k]})P^{\mu}(dy_{[n,n+k]}|Y_{[0,n-1]}) \\
&=\int_{\mathcal{Y}^{k+1}\times \mathcal{X}}g(y_{[n,n+k]})P^{\mu}(d(y_{[n,n+k]},x_{n})|Y_{[0,n-1]})
\end{align*}
we then apply the chain rule for conditional probability measures and we have
\begin{align*}
&\int_{\mathcal{X}}\int_{\mathcal{Y}^{k+1}}g(y_{[n,n+k]})P^{\mu}(dy_{[n,n+k]}|x_{n},Y_{[0,n-1]})\pi^{\mu}_{n-}(dx_{n})
\end{align*}
Since $\{(X_{n},Y_{n})\}_{n=0}^{\infty}$ is a Markov chain chain, $Y_{[n,n+k]}$ is conditionally independent of $Y_{[0,n-1]}$ given $X_{n}$. Additionally, the prior does not determine the conditional measure, therefore we have 
\begin{align*}
\int_{\mathcal{X}}\int_{\mathcal{Y}^{k+1}}g(y_{[n,n+k]})P(dy_{[n,n+k]}|x_{n})\pi^{\mu}_{n-}(dx_{n})
\end{align*}
where we do not include a prior in the superscript of the conditional measure, since the conditional measure is the same regardless of the prior.
\end{proof}

\begin{corollary}\label{H_Q_one_step}
Let $g$ be a bounded and measurable function on $(\mathcal{Y},\mathcal{B}(\mathcal{Y}))$. For any prior $\mu$ we have
\begin{align}
\int_{\mathcal{Y}}g(y_{n})P^{\mu}(dy_{n}|X_{n}=x)=\int_{\mathcal{Z}}g(h_{x}(z))Q(dz)\label{H_Q}
\end{align}
\end{corollary}
\begin{proof}
$Z$ is a random variable on the probability space $(\mathcal{Z},\mathcal{B}(\mathcal{Z}),Q)$ and $Y_{n}$ exists on the measurable space $(\mathcal{Y},\mathcal{B}(\mathcal{Y}))$. Then, for every fixed choice of $X_{n}=x$ we have that $Y_{n}$ is a fixed function of $Z$, that is $Y_{n}=h_{x}(Z)$. For any set $A \in \mathcal{B}(\mathcal{Y})$ we have $P^{\mu}(Y_{n} \in A|X_{n}=x)=Q(h_{x}^{-1}(A))$. Yet this means that $P^{\mu}(Y_{n} \in \cdot|X_{n}=x)$ is exactly the pushforward measure of $Q$ under the mapping $h_{x}$, call this measure $h_{x}Q(A)=Q(h_{x}^{-1}(A))$. We then have:
\begin{align*}
\int_{\mathcal{Y}}g(y)h_{x}Q(dy))=\int_{\mathcal{Z}}g(h_{x}(z))Q(dz).
\end{align*}


\end{proof}

Notice that the inner integral in the RHS of Equation~(\ref{bayes_RHS}) is a function of $x$. The LHS  is then the term considered in the total variation merging of the predictive measures of the measurement sequences, while the RHS is the term considered in the weak merging of the one-step predictor. We can then leverage Blackwell and Dubin's theorem to arrive at a sufficient condition for weak merging of the one-step predictor. Theorem \ref{cont_bounded} is closely related to \cite[Prop. 3.11]{van2009uniform} and its proof is in essence a sufficient condition for uniform observability (of the predictor).

\begin{proof}[\textbf{Proof of Theorem \ref{cont_bounded}}]~\\
Fix any $f \in C_{b}(\mathcal{X})$ and $\epsilon>0$. We wish to show that $\exists N$ such that $\forall n>N$, $$\left|\int fd\pi_{n-}^{\mu}-\int fd\pi_{n-}^{\nu} \right|<\epsilon$$
\sy{By observability for the fixed $f$, (\ref{Nstepobs}) holds for some $N'+1$}. Therefore we can find some $g$ with $\|g\|_{\infty}<\infty$ such that $$\tilde{f}(x)=\int_{\mathcal{Y}^{N'+1}}g(y_{[1,1+N']})P(dy_{[1,1+N']}|X_{1}=x)$$  and $\|f-\tilde{f}\|_{\infty}<\frac{\epsilon}{3}$. Conditioned on the value of $X_{n}=x$ and since the noise is i.i.d, the conditional channel $Y_{[n,n+N']}|X_{n}$ is time invariant, so it holds that
\begin{align*}
\tilde{f}(x)=\int_{\mathcal{Y}^{N'+1}}g(y_{[n,n+N']})P(dy_{[n,n+N']}|X_{n}=x)
\end{align*}
is the same regardless of the choice of $n$.
 Then we have
\begin{align}
& \left|\int fd\pi_{n-}^{\mu}-\int fd\pi_{n-}^{\nu} \right| \nonumber \\
&\leq \left|\int \tilde{f}d\pi^{\mu}_{n-}-\int \tilde{f}d\pi_{n-}^{\nu} \right|+\left|\int (f-\tilde{f})d\pi^{\mu}_{n-} \right|  \nonumber \\
& +\left|\int (f-\tilde{f})d\pi^{\nu}_{n-} \right|\label{approx_eqn}
\end{align}

Now, by assumption $\|f-\tilde{f}\|_{\infty}<\frac{\epsilon}{3}$ therefore the last two terms are less than $\frac{2}{3}\epsilon$. We then apply Lemma \ref{bayes_rearrange} and we have
\begin{align*}
&\left|\int \tilde{f}d\pi^{\mu}_{n-}-\int \tilde{f}d\pi_{n-}^{\nu} \right|+\frac{2}{3}\epsilon\\
&=|\int_{\mathcal{Y}^{N'+1}} g(y_{[n,n+N']})P^{\mu}(dy_{[n,n+N']} |Y_{[0,n-1]}) \\
& \quad -\int_{\mathcal{Y}^{N'+1}}g(y_{[n,n+N']})P^{\nu}(dy_{[n,n+N']}|Y_{[0,n-1]})|+\frac{2}{3}\epsilon
\end{align*}

By Assumption \ref{absContMeas}, we have $P^{\mu}(Y_{[0,\infty)}\in \cdot) \ll P^{\nu}(Y_{[0,\infty)} \in \cdot)$. Then via a classic result by Blackwell and Dubins \cite{blackwell1962merging}, we have that $P^{\mu}(Y_{[n,n+N']} \in \cdot |Y_{[0,n-1]})$ and $P^{\nu}(Y_{[n,n+N']} \in \cdot|Y_{[0,n-1]})$ merge in total variation $P^{\mu}$ a.s. as $n \to \infty$. Define $\tilde{g}=\frac{g}{\|g\|_{\infty}}$. Then $\exists~N\in \mathbb{N}$ such that $\forall n>N$,
\begin{align*}
&|\int_{\mathcal{Y}^{N'+1}} \tilde{g}(y_{[n,n+N']})P^{\mu}(dy_{[n,n+N']} |Y_{[0,n-1]}) \\
& -\int_{\mathcal{Y}^{N'+1}}\tilde{g}(y_{[n,n+N']})P^{\nu}(dy_{[n,n+N']}|Y_{[0,n-1]})| <\frac{\epsilon}{3\|g\|_{\infty}}
\end{align*}
we then have:
\begin{align*}
&|\int_{\mathcal{Y}^{N'+1}} g(y_{[n,n+N']})P^{\mu}(dy_{[n,n+N']} |Y_{[0,n-1]}) \\
& -\int_{\mathcal{Y}^{N'+1}}g(y_{[n,n+N']})P^{\nu}(dy_{[n,n+N']}|Y_{[0,n-1]})|+\frac{2}{3}\epsilon\\
&\leq \|g\|_{\infty}\frac{\epsilon}{3 \|g\|_{\infty}}+\frac{2}{3}\epsilon=\epsilon
\end{align*}
therefore, since $f$ and $\epsilon$ are arbitrary we have for any $f \in C_{b}(S)$:
$\lim_{n \to \infty}\left|\int fd\pi_{n-}^{\mu}-\int fd\pi_{n-}^{\nu} \right|=0,$ which means $\pi_{n-}^{\mu}$ and $ \pi_{n-}^{\nu}$ merge weakly.
\end{proof}

\subsection{Weak Filter Stability: Proof of Theorem \ref{Ali_Theorem}}\label{proofAli}
Here we will utilize results from \cite{KSYWeakFellerSysCont}. This paper was concerned with a different topic than filter stability, namely the weak Feller property of the ``filter update'' kernel. That is, one can view the filter $\pi_{n}^{\mu}$ and the measurement $Y_{n}$ as its own Markov chain $\{(\pi_{n}^{\mu},Y_{n})\}_{n=0}^{\infty}$ which takes values in $\mathcal{P}(\mathcal{X}) \times \mathcal{Y}$. The filter update kernel is the transition kernel of this Markov chain. We will not study this kernel, but some of the analysis in \cite{KSYWeakFellerSysCont} is useful in providing concise arguments to connect the filter to the predictor.

\begin{proof}[\textbf{Proof of Theorem \ref{Ali_Theorem}}]~\\
Begin by assuming that the predictor merges weakly almost surely. As is argued in \cite{KSYWeakFellerSysCont}, one can view the filter $\pi_{n}^{\mu}$ as a function of $\pi_{n-1}^{\mu}$ (the previous filter) and the current observation $Y_{n}=y_{n}$, that is $\pi^{\mu}_{n}=F(\pi^{\mu}_{n-1},y_{n})$. Pick any continuous and bounded function $f$, we have
\begin{align}
&E^{\mu}[|\int_{\mathcal{X}}f(x)\pi_{n}^{\mu}(dx)-\int_{\mathcal{X}}f(x)\pi_{n}^{\nu}(dx)|] \nonumber \\
=&E^{\mu}[E^{\mu}[| \int_{\mathcal{X}}f(x)F(\pi_{n-1}^{\mu},y_{n})(dx) \nonumber \\
& \quad -\int_{X}f(x)F(\pi_{n-1}^{\nu},y_{n})(dx)  |  |Y_{[0,n-1]}  ] ] \label{ali_eq1}
\end{align}
Now, define the set $I^{+}(y_{[0,n-1]})\subset \mathcal{Y}$ as:
\begin{align*}
I^{+}(y_{[0,n-1]})&= \{y_{n}\in \mathbb{Y} |\int_{\mathcal{X}}f(x)F(\pi_{n-1}^{\mu},y_{n})(dx)  \\
& >\int_{X}f(x)F(\pi_{n-1}^{\nu},y_{n})(dx) \}
\end{align*}
where the argument $y_{[0,n-1]}$ is the sequence on which the previous filters $\pi_{n-1}^{\mu}$ and $\pi_{n-1}^{\nu}$ are realized. Define the complement of this set as $I^{-}(y_{[0,n-1]})$. Then for every fixed realization $y_{[0,n-1]}$ we can break the inner expectation \sy{in (\ref{ali_eq1})} (which is an integral) into two parts and follow the analysis in \cite[Equation 4]{KSYWeakFellerSysCont} together with Theorem 8.6.2 in \cite{Bogachev} to arrive at the conclusion.

\end{proof}


\subsection{Local Observability: Proof of Theorem \ref{local_thm}  and \ref{local_thmRelaxed}}\label{prooflocal_thmRelaxed}

The idea of local observability is the shift some of the burden of approximating the signal $f$. When we work with a function $$\tilde{f}(x)=\int_{\mathcal{Y}^{N'+1}}g(y_{[n,n+N']})P(dy_{[n,n+N']}|X_{n}=x)$$ the result is the terms seen in equation (\ref{approx_eqn}). The first term is dealt with by Blackwell and Dubin's theorem, so we must make sure the second and third term can be made arbitrarily small. For any set $K$ we can write
\begin{align*}
&\left|\int (f-\tilde{f})d\pi^{\nu}_{n-} \right|\leq \sup_{x \in K}|f(x)-\tilde{f}(x)|\pi_{n-}^{\nu}(K) \\
&+\sup_{x \not\in K} |f(x)-\tilde{f}(x)|\pi_{n-}^{\mu}(K^{C})
\end{align*}
in the previous result we bounded this by simply approximating $f$ well over the whole space. Instead, we can choose a $K$ where $\tilde{f}$ approximates $f$ well over $K$ and $\pi_{n-}^{\nu}(K^{C})$ makes the other term arbitrarily small. Furthermore, by taking advantage of the full supremum of total variation we can work with a series of uniformly bounded functions $\tilde{f}_{n}$ and shifting sets $K_{n}$ that change with $n$.

\label{local_proof}
\begin{proof}[\textbf{Proof of Theorem \ref{local_thm}}]
Pick any continuous and bounded function $f$ and any $\epsilon>0$. Fix any sequence of observations $y_{[0,\infty)}$ where the predictors $\pi_{n-}^{\mu}$ and $\pi_{n-}^{\nu}$ are well defined and maintain this sequence for the remainder of the proof. Then consider
\begin{align*}
\lim_{n \to \infty}\left|\int_{\mathcal{X}}f(x)\pi_{n-}^{\mu}(dx)-\int_{\mathcal{X}}f(x)\pi_{n-}^{\nu}(dx) \right|
\end{align*}
For any function  series of functions $\tilde{f}_{n}$ of $x$ we have an upper bound
\begin{align*}
&\lim_{n \to \infty}\left|\int_{\mathcal{X}}\tilde{f}_{n}(x)\pi_{n-}^{\mu}(dx)-\int_{\mathcal{X}}\tilde{f}_{n}(x)\pi_{n-}^{\nu}(dx) \right| \\
&+\left|\int_{\mathcal{X}}(f-\tilde{f}_{n})(x)\pi^{\mu}_{n-}(dx) \right|+\left|\int_{\mathcal{X}}(f-\tilde{f}_{n})(x)\pi^{\nu}_{n-}(dx) \right|
\end{align*}
By assumption of $K$ local predictability, we have a compact sets $K_{n}=K+a_{n}$ where $\pi_{n-}^{\mu}(K_{n})=1$ for every $\mu \ll \nu$ and every $n$.

By $K$ local observability, we can find a uniformly bounded series of functions $g_{n} \leq M$ where
\begin{align*}
\tilde{f}_{n}(x)=\int_{\mathcal{Z}}g_{n}(h(x,z))Q(dz)\\
\sup_{x \in K_{n}} \left|f(x)-\tilde{f}_{n}(x) \right|\leq \frac{\epsilon}{3}
\end{align*}
then for the two approximation terms we have
\begin{align*}
& \left|\int_{\mathcal{X}}(f-\tilde{f}_{n})(x)\pi^{\nu}_{n-}(dx) \right| \\
&\leq \sup_{x \in K_{n}}|f(x)-\tilde{f}_{n}(x)|\pi_{n-}^{\nu}(K_{n})+\sup_{x \not\in K_{n}} |f(x)-\tilde{f}_{n}(x)|\pi_{n-}^{\mu}(K_{n}^{C})\\
&\leq \frac{\epsilon}{3}
\end{align*}we then have
\begin{align*}
&\lim_{n \to \infty}\left|\int_{\mathcal{X}}\tilde{f}_{n}(x)\pi_{n-}^{\mu}(dx)-\int_{\mathcal{X}}\tilde{f}_{n}(x)\pi_{n-}^{\nu}(dx) \right|+\frac{2}{3}\epsilon\\
=&\lim_{n \to \infty} \bigg|\int_{\mathcal{Y}}g_{n}(y_{n})P^{\mu}(dy_{n}|y_{[0,n-1]}) \\
& \qquad -\int_{\mathcal{Y}}g_{n}(y_{n})P^{\nu}(dy_{n}|y_{0,n-1]}) \bigg| +\frac{2}{3}\epsilon
\end{align*} 
we must appeal to the full uniform bound of the Blackwell and Dubins theorem, which was not required in the proof of Theorem \ref{cont_bounded}. The full statement of the Blackwell and Dubins theorem tells us that
\begin{align}
&\lim_{n \to \infty} \sup_{\|g\|\leq 1} \nonumber \\
&\left|\int_{\mathcal{Y}}g(y_{n})P^{\mu}(dy_{n}|y_{[0,n-1]})-\int_{\mathcal{Y}}g(y_{n})P^{\nu}(dy_{n}|y_{0,n-1]}) \right|=0,\label{bd_thm}
\end{align}
where the supremum is taken over measurable functions $g$. Thus, for any fixed measurable and bounded function $g$, we have that
\begin{align*}
\left|\int_{\mathcal{Y}}g(y_{n})P^{\mu}(dy_{n}|y_{[0,n-1]})-\int_{\mathcal{Y}}g(y_{n})P^{\nu}(dy_{n}|y_{0,n-1]}) \right|
\end{align*}
converges to $0$ as $n \to \infty$; this was the form of the statement utilized in the proof of Theorem \ref{cont_bounded}. However, if we have a sequence of measurable functions $g_{n}$ with a uniform bound, $g_{n}\leq M~\forall~n \in \mathbb{N}$, then the supremum in (\ref{bd_thm}) allows us to make a uniform claim about the convergence to zero of the sequence,
\begin{align*}
\left|\int_{\mathcal{Y}}g_{n}(y_{n})P^{\mu}(dy_{n}|y_{[0,n-1]})-\int_{\mathcal{Y}}g_{n}(y_{n})P^{\nu}(dy_{n}|y_{0,n-1]}) \right|,
\end{align*}
and this completes the proof. 
\end{proof}

\begin{proof}[\textbf{Proof of Theorem \ref{local_thmRelaxed}}]\label{local_relaxed_proof}
Fix any $f$ and any $\epsilon$. We begin from the upper bound used previously
\begin{align*}
&\lim_{n \to \infty}\left|\int_{\mathcal{X}}\tilde{f}_{n}(x)\pi_{n-}^{\mu}(dx)-\int_{\mathcal{X}}\tilde{f}_{n}(x)\pi_{n-}^{\nu}(dx) \right| \\
&+\left|\int_{\mathcal{X}}(f-\tilde{f}_{n})(x)\pi^{\mu}_{n-}(dx) \right|+\left|\int_{\mathcal{X}}(f-\tilde{f}_{n})(x)\pi^{\nu}_{n-}(dx) \right|
\end{align*}
for some series of functions $\tilde{f}_{n}$.

By local predictability, the shifted predictors are a tight family. Therefore for any $\epsilon'$ we have a series of compact sets $K_{n}=K'+a_{n}$ such that $\pi_{n-}^{\nu}(K_{n})\geq 1-\epsilon'$ for any $\mu \ll \nu$ and any $n$. 

The proof then proceeds similarly as that of Theorem \ref{local_thm}

\end{proof}

\subsection{Predictor Merging in Total Variation: Proof of Theorem \ref{weak_to_total}}\label{proofweak_to_total}

We now  extend our results from weak merging to total variation. We first state the following supporting results.
\begin{lemma}\label{weakPriorweakFilter}
The (measurement-update) map: 
\begin{align*}
 (\pi_{n_-},y) \mapsto \pi_{n} \quad: \quad \pi_{n}(\cdot):=E_{\pi_{n_-}} [ 1_{X_n \in \cdot } | Y_n=y]
\end{align*}
which maps from $\mathcal{P}(\mathcal{X}) \times \mathbb{Y}$ to $\mathcal{P}(\mathcal{X})$ is weakly continuous in $\pi_{n_-}$ for almost every $y$, provided that $g(x,y)$ is positive, bounded and continuous in $x$ for every fixed $y$.
\end{lemma}

\textbf{Proof.} Consider a continuous and bounded $f$ and let $\pi^m_{n_-} \to \pi_{n_-}$ weakly. Then,
\begin{align}
E_{\pi^m_{n_-}}[f(x_n)|Y_n=y_n] &= \int f(x_{n}) \frac{g(x_{n},y_{n}) \pi^m_{n_-}(dx_n)} {\int_{\mathcal{X}} g(x_{n},y_{n}) \pi^m_{n_-}(dx_n)} \nonumber \\
&=  \frac{ \int f(x_{n}) g(x_{n},y_{n}) \pi^m_{n_-}(dx_n)} {\int_{\mathcal{X}} g(x_{n},y_{n}) \pi^m_{n_-}(dx_n)} \nonumber
\end{align}
Since $g(\cdot,y_n)$ is bounded and continuous, both the numerator and the denominator converge. \qed

\begin{lemma}\label{weakPostTVPredictor}
Let $T(dx_1|x) = t(x_1,x) \phi(dx_1)$ where $t$ is continuous in $x$ for every $x_1$. Then, the (time-update) map: 
\[(\pi_{n}) \mapsto \pi_{{n+1}_-} \quad : \quad  \pi_{{n+1}_-}(\cdot):=\int T(\cdot|x_n) \pi_{n}(dx_n)\] 
which maps from $\mathcal{P}(\mathcal{X})$ to $\mathcal{P}(\mathcal{X})$ is so that
if $\pi^{\nu}_n \to \pi^{\mu,}_n$ weakly then $\pi^{\nu}_{{n+1}_-} \to \pi^{\mu}_{{n+1}_-}$ in total variation.
\end{lemma}

\begin{proof}
We will build on Scheff\'e's Lemma \cite{Bil86}. For every given history, we have
\[\pi^{\nu}_{{n+1}_-}(dx_{n+1}) =  \int T(dx_{n+1}|x_n) \pi^{\nu}_{n}(dx_n)\]
Now, $\int T(dx_{n+1}|x_n)$ is so that,
\[ \int  t(x_{n+1},x_n) \phi(dx_{n+1}) \pi^{m}_n(dx_n) \to \int  t(x_{n+1},x_n) \phi(dx_{n+1}) \pi_n(dx_n) \]
in total variation since for every fixed $z$, the Radon-Nikodym derivative (density) with respect to $\phi$
\[\frac{ \int  t(x_{n+1},x_n) \phi(\cdot) \pi^{m}_n(dx_n)} {d\phi}(z) =   \int  t(z,x_n) \pi^{m}_n(dx_n)\]
satisfies pointwise convergence
\[ \int  t(z,x_n) \pi^{\nu}_n(dx_n) \to \int  t(z,x_n) \pi^{\mu}_n(dx_n) \]
and Scheff\'e's lemma implies that convergence is in total variation. Now, we can apply the result to the sequence $\pi^{\nu}_n$ converging to $\pi^{\mu}_n$.
\end{proof}

\begin{proof}\textbf{Proof of Theorem \ref{weak_to_total}(i)} 


Under Assumption \ref{lebesgue_cont_control2}, the proof follows from Lemma \ref{weakPriorweakFilter} and \ref{weakPostTVPredictor}. While in Lemma \ref{weakPriorweakFilter} and \ref{weakPostTVPredictor} we consider convergence (and not merging), we note that the proof of Lemma \ref{weakPriorweakFilter} also implies weak merging of the posteriors as the priors weakly merge, and by considering the signed measure $\pi^{\nu,\gamma}_n - \pi^{\mu,\gamma}_n$ in the proof of Lemma \ref{weakPostTVPredictor}, total variation merging is a result of a generalized Scheff\'e's lemma \cite[Theorem 2.8.9]{Bogachev}.
\end{proof}


\begin{lemma}\label{dominating measure}
Let $\exists$ some measure $\bar{\mu}$ such that $T(\cdot|x) \ll \bar{\mu}$ for every $x \in \mathcal{X}$. Then we have that $\pi_{n-}^{\mu},\pi_{n-}^{\nu} \ll \bar{\mu}$ for every $n \in  \mathbb{N}$
\end{lemma}
\begin{proof}
For all $n\geq 1$ we have
\begin{align*}
\pi_{n-}^{\mu}(A)&=\int_{\mathcal{X}}T(A|x)\pi_{n-1}^{\mu}(dx)
=\int_{\mathcal{X}}\int_{A}\frac{dT(\cdot|x)}{d\bar{\mu}}(a)\bar{\mu}(da)\pi_{n-1}^{\mu}(dx)\\
&=\int_{A}\left(\int_{\mathcal{X}}\frac{dT(\cdot|x)}{d\bar{\mu}}(a)\pi_{n-1}^{\mu}(dx)\right)\bar{\mu}(da)
\end{align*}
where we have applied Fubini's theorem in the final equality. Therefore $\pi_{n-}^{\mu}$ is absolutely continuous with respect to $\bar{\mu}$ for every $n\geq 1$.
\end{proof}
 
\begin{lemma}\label{equicontinuity} 
Let Assumption \ref{lebesgue_cont} hold and let $f^{\mu}_{n-}$ denote the density function of $\pi_{n-}^{\mu}$. Fix any sequence of measurements $y_{[0,\infty)}$ and denote the collection of probability density functions $\mathscr{F}^{\mu}=\{f_{n-}^{\mu}|n\in \mathbb{N}\},\mathscr{F}^{\nu}=\{f_{n-}^{\nu}|n\in \mathbb{N}\}$. Then $\mathscr{F}^{\mu}, \mathscr{F}^{\nu}$ are uniformly bounded equicontinuous families.
\end{lemma}
\begin{proof}
As we see from Lemma \ref{dominating measure},
\begin{align*}
f_{n-}^{\mu}(x_{n})=\frac{d \pi_{n-}^{\mu}}{d\phi}(x_{n})=\int_{\mathcal{X}}t(x_{n}|x_{n-1})\pi_{n-1}^{\mu}(dx_{n-1})
\end{align*}
Where $t(\cdot|x)$ is the Radon Nikodym derivative of $T(\cdot|x)$ with respect to our dominating measure $\phi$ and $d(\cdot,\cdot)$ will represent the metric on $\mathcal{X}$ (recall $\mathcal{X}$ is a complete, separable, metric space). We require $\forall \epsilon>0,~x^{*}\in \mathcal{X}~\exists~\delta>0$ such that $\forall~d(x,x^{*})<\delta$, $\forall n \in \mathbb{N}$ we have $|f_{n-}^{\mu}(x)-f_{n-}^{\mu}(x^{*})|<\epsilon$. By Assumption \ref{lebesgue_cont}, clearly $f_{n-}^{\mu}$ is uniformly bounded since $t$ is uniformly bounded. Then, for any $\epsilon>0$, $\forall x^{*} \in \mathcal{X}$ we can find a $\delta>0$ such that $\forall x_1\in \mathcal{X}$, $|t(x_{2}|x_{1})-t(x^{*}|x_{1})|<\epsilon$ when $d(x_{2},x^{*})<\delta$ . Now, assume $d(x_{2},x^{*})<\delta$, we have

\begin{align*}
|f_{n-}^{\mu}(x_{2})-f_{n-}^{\mu}(x^{*})|&=\left|\int_{\mathcal{X}}t(x_{2}|x_{1})-t(x^{*}|x_{1})d\pi_{n-}^{\mu}(dx_{1})\right|\\
&\leq \int_{\mathcal{X}}|t(x_{2}|x_{1})-t(x^{*}|x_{1})|d\pi_{n-}^{\mu}(x_{1})
\leq \epsilon
\end{align*}

which proves that $\mathscr{F}^{\mu}$ and $\mathscr{F}^{\nu}$ are uniformly bounded and equicontinuous families.
\end{proof}
\begin{proof}[\textbf{Proof of Theorem \ref{weak_to_total}}(ii)]
By assumption we have weak stability of the predictor $P^{\mu}$ a.s.. Then there exists a set of measure sequences $B\subset \mathcal{Y}^{\mathbb{Z}_{+}}$ with $P^{\mu}(B)=1$. For each measurement sequence $y_{[0,\infty]} \in B$, we have that the predictor realizations $\pi_{n-}^{\mu}$ and $\pi_{n-}^{\nu}$ merge in the weak sense. We will choose a general measurement sequence $y_{[0,\infty]} \in B$ and fix this sequence for the remainder of the proof.
Via Lemma \ref{dominating measure}, and \ref{equicontinuity}, $\mathscr{F}^{\mu}$ and $\mathscr{F}^{\nu}$ are uniformly bounded and equicontinuous families. Let $\mathscr{F}^{\mu-\nu}=\{f_{n}|f_{n}=f_{n-}^{\mu}-f_{n-}^{\nu}]\}$, then the sequence $\{f_{n}\}_{n=1}^{\infty}$ is a uniformly bounded and equicontinuous class of integrable functions. As in the proof of \cite[Lemma 2]{linder2014optimal}, now pick a sequence of compact sets $K_{j}\subset \mathcal{X}$ such that $K_{j} \subset K_{j+1}$. By the Arzela-Ascoli theorem \cite{rudin2006real}, for any subsequence we can find further subsequences $f_{n_{k}^{j}}$ such that 
\begin{align*}
\lim_{k \to \infty} \sup_{x \in K_{j}}|f_{n_{k}^{j}}(x)-f^{j}(x)|=0
\end{align*}
for some continuous function $f^{j}:K_{j} \to [0,\infty)$. Via the $K_{j}$ being nested, we can have $\{f_{n_{k}^{j+1}}\}$ be a subsequence of $\{f_{n_{k}^{j}}\}$, and therefore $f^{j+1}=f^{j}$ over $K_{j}$. Then define the function $\tilde{f}$ on $\mathcal{X}$ by $\tilde{f}(x)=f^{j}(x),x \in K_{j}$. Using Cantor's diagonal method, we can find an increasing sequence of integers $\{m_{i}\}$ which is a subsequence of $\{n_{k}^{j}\}$ for every $j$. Therefore
\begin{align*}
\lim_{i \to \infty} f_{m_{i}}(x)=\tilde{f}(x)~~\forall x \in \mathcal{X}
\end{align*}
and the convergence is uniform over each $K_{j}$ and $\tilde{f}$ is continuous. Now, $f_{m_{i}}$ converges weakly to the zero measure by assumption, and via uniform convergence for any Borel set $\mathcal{B}$ we have \[\int_{\mathcal{B}}f_{m_{i}}(x)dx \to \int_{\mathcal{B}}\tilde{f}(x)dx,\] i.e. setwise convergence. 
Yet this implies weak convergence, so $\tilde{f}=0$ almost everywhere, yet $\tilde{f}$ is continuous so it is 0 everywhere. Now, via Prokhorov theorem (Theorem 8.6.2 in \cite{Bogachev}) we have that $\mathscr{F}^{\mu-\nu}$ is a tight family. Therefore, for every $\epsilon>0$ we can find a compact set $K_{\epsilon}$ such that 
\[|\pi_{n-}^{\mu}-\pi_{n-}^{\nu}|(\mathcal{X}\setminus K_{\epsilon})<\epsilon~\forall~n \in \mathbb{N}.\]
 then we have
\begin{align*}
&\lim_{i \to \infty} \|\pi_{m_{i}-}^{\mu}-\pi_{m_{i}-}^{\nu}\|_{TV}\leq \lim_{i \to \infty}|\pi_{m_{i}-}^{\mu}-\pi_{m_{i}-}^{\nu}|(\mathcal{X}\setminus K_{\epsilon}) \\
& \quad \quad +|\pi_{m_{i}-}^{\mu}-\pi_{m_{i}-}^{\nu}|( K_{\epsilon})\\
&\leq \lim_{i \to \infty} \sup_{\|g\|_{\infty}\leq 1}\left|\int_{K_{\epsilon}}g(x)f_{m_{i}}(x)dx \right|+\epsilon\\
&\leq \lim_{i \to \infty}\sup_{\|g\|_{\infty}\leq 1}\left| \int_{K_{\epsilon}}g(x)(\tilde{f}-f_{m_{i}})(x)dx\right|+\left|\int_{K_{\epsilon}}g(x)\tilde{f}(x)dx \right| +\epsilon\\
&\leq \lim_{i \to \infty} \|\tilde{f}-f_{m_{i}}\|_{\infty}\phi(K_{\epsilon})+\epsilon
\end{align*}
since we have already argued $\tilde{f}=0$. Now, over the compact set $K_{\epsilon}$, $f_{m_{i}}$ converges to $\tilde{f}$ uniformly, therefore $\exists N$ such that $\forall k>N$, $\|\tilde{f}-f_{n_{k}}\|_{\infty}<\frac{\epsilon}{\phi({K}_{\epsilon})}$. We then conclude that
\begin{align*}
\lim_{i \to \infty} \|\pi_{m_{i}-}^{\mu}-\pi_{m_{i}-}^{\nu}\|_{TV}=0
\end{align*}
Thus, for every subsequence of $\{f_{n}\}_{n=1}^{\infty}$, we can find a subsequence that converges in total variation, which implies that the original sequence converges in total variation.
\end{proof}


\subsection{Filter Merging in Total Variation: Proof of Theorem \ref{tv_equal_placeholder}}\label{prooftv_equal_placeholder}


For completeness, in Section \ref{appenSupTV} some supporting results are presented. 

\begin{proof}[\textbf{Proof of Theorem \ref{tv_equal_placeholder}}]
The sigma fields $\mathcal{F}^{\mathcal{X}}_{n,\infty} \vee \mathcal{F}^{\mathcal{Y}}_{0,\infty}$ are a decreasing sequence, that is $\mathcal{F}^{\mathcal{X}}_{n+1,\infty} \vee \mathcal{F}^{\mathcal{Y}}_{0,\infty}~\subset~\mathcal{F}^{\mathcal{X}}_{n,\infty} \vee \mathcal{F}^{\mathcal{Y}}_{0,\infty}$. Therefore, when we take their intersection, removing the first or largest sigma field $\mathcal{F}^{\mathcal{X}}_{0,\infty} \vee \mathcal{F}^{\mathcal{Y}}_{0,\infty}$ from the intersection of a deceasing set of sigma fields does not change the overall intersection. From Lemma \ref{filt_iff} and \ref{pred_sigma_field_condition}, it is clear that the two conditions for merging in total variation in expectation are equivalent since the sigma fields on the LHS  of Equation (\ref{sigma_field_condition}) and (\ref{pred_sigma_eqn}) are equal.
\end{proof}

We have now established that the filter merges in total variation in expectation, but we would like to extend this result to almost surely. By a simple application of Fatou's lemma, we can argue the liminf of the total variation of the filter is zero $P^{\mu}$ a.s.. Hence if the limit exists, it must be zero, yet it is not immediate that the limit will exist. This leads to the following. 

\begin{theorem}\cite[p. 572]{van2008discrete} \label{almost_sure_stability}
Assume the filter is stable in total variation in expectation. Then the filter is stable in total variation $P^{\mu}$ a.s.
\end{theorem}


\subsection{Relative Entropy Merging: Proof of Theorem \ref{re_tv_equal}}\label{proofre_tv_equal}

We will now show that the relative entropy merging of the filter is essentially equivalent to merging in total variation in expectation.  Via Lemma \ref{filter_rd_derivative} and \ref{predictor_rd_derivative}, it is clear that the filter and predictor admit Radon-Nikodym derivatives. Therefore, working with $D(\pi_{n}^{\mu}\| \pi_{n}^{\nu})$ and $D(\pi_{n-}^{\mu}\|\pi_{n-}^{\nu})$ is well defined.
A well known result for relative entropy is the chain rule \cite[Theorem 5.3.1]{gray2011entropy}:
\begin{lemma}\label{chain rule}
For joint measures $P,Q$ on random variables $X,Y$ we have
\begin{align*}
D(P(X,Y)\|Q(X,Y))=D(P(X)\|Q(X))+D(P(Y|X)\|Q(Y|X))
\end{align*}
\end{lemma}
Note for two sigma fields $\mathcal{F}$ and $\mathcal{G}$ and two joint measures $P$ and $Q$ on $\mathcal{F} \vee \mathcal{G}$ one could also express this relationship as
\begin{align}\label{chainRInfo}
D(P|_{\mathcal{F} \vee \mathcal{G}}\|Q|_{\mathcal{F} \vee \mathcal{G}})&=D(P|_{\mathcal{F}}\|Q|_{\mathcal{F}})+D(P|_{\mathcal{G}}|\mathcal{F}\|Q|_{\mathcal{G}}|\mathcal{F})
\end{align}

\begin{proof}[\textbf{Proof of Theorem \ref{re_tv_equal}}]
First assume the filter is stable in relative entropy. Since the square root function is continuous and convex, we have
\begin{align*}
&0=\lim_{n \to \infty} \sqrt{\frac{2}{\log(e)}E^{\mu}[D(\pi^{\mu}_{n}\|\pi_{n}^{\nu})]} \\
&\quad \geq \lim_{n \to \infty} E^{\mu}\left[\sqrt{\frac {2}{\log(e)}D(\pi^{\mu}_{n}\|\pi_{n}^{\nu})} \right]
\end{align*}
where we have applied Jensen's inequality. We then apply Pinsker's inequality and we have $\lim_{n \to \infty}E^{\mu}[\|\pi_{n}^{\mu}-\pi_{n}^{\nu}\|_{TV}]=0$.

For the converse direction, by chain rule (\ref{chainRInfo}), it is clear that
\begin{align*}
&E^{\mu}[D(\pi_{n}^{\mu}\|\pi_{n}^{\nu})]=D(P^{\mu}|_{\mathcal{F}^{\mathcal{X}}_{n}}|\mathcal{F}^{\mathcal{Y}}_{0,n}\|P^{\nu}|_{\mathcal{F}^{\mathcal{X}}_{n}}|\mathcal{F}^{\mathcal{Y}}_{0,n})\\
&\quad=D(P^{\mu}|_{\mathcal{F}^{\mathcal{X}}_{n} \vee \mathcal{F}^{\mathcal{Y}}_{0,n}}\|P^{\nu}|_{\mathcal{F}^{\mathcal{X}}_{n} \vee \mathcal{F}^{\mathcal{Y}}_{0,n}})-D(P^{\mu}|_{\mathcal{F}^{\mathcal{Y}}_{0,n}}\|P^{\nu}|_{\mathcal{F}^{\mathcal{Y}}_{0,n}})
\end{align*}
by the Markov Property we have $X_{[0,n-1]},Y_{[0,n-1]}$ and $X_{[n+1,\infty)},Y_{[n+1,\infty)}$ are conditionally independent given $X_{n},Y_{n}$ therefore we have:
\begin{align*}
& D(P^{\mu}|_{\mathcal{F}^{\mathcal{X}}_{n} \vee \mathcal{F}^{\mathcal{Y}}_{0,n}}\|P^{\nu}|_{\mathcal{F}^{\mathcal{X}}_{n} \vee \mathcal{F}^{\mathcal{Y}}_{0,n}}) \\
&\quad =D(P^{\mu}|_{\mathcal{F}^{\mathcal{X}}_{n,\infty} \vee \mathcal{F}^{\mathcal{Y}}_{0,\infty}}\|P^{\nu}|_{\mathcal{F}^{\mathcal{X}}_{n,\infty} \vee \mathcal{F}^{\mathcal{Y}}_{0,\infty}})
\end{align*}
Then $\mathcal{F}^{\mathcal{X}}_{n,\infty} \vee \mathcal{F}^{\mathcal{Y}}_{0,\infty}$ is a decreasing sequence of sigma fields. By \cite[Theorem 2]{BarronSorrento} we have that if the relative entropy is ever finite, the limit of the relative entropy  restricted to these sigma fields is the relative entropy restricted to the intersection of the decreasing fields, that is
\begin{align*}
&\lim_{n \to \infty}D(P^{\mu}|_{\mathcal{F}^{\mathcal{X}}_{n,\infty} \vee \mathcal{F}^{\mathcal{Y}}_{0,\infty}}\|P^{\nu}|_{\mathcal{F}^{\mathcal{X}}_{n,\infty} \vee \mathcal{F}^{\mathcal{Y}}_{0,\infty}}) \nonumber \\
&=D(P^{\mu}|_{\bigcap_{n \geq 0}\mathcal{F}^{\mathcal{X}}_{n,\infty} \vee \mathcal{F}^{\mathcal{Y}}_{0,\infty}}\|P^{\nu}|_{\bigcap_{n \geq 0}\mathcal{F}^{\mathcal{X}}_{n,\infty} \vee \mathcal{F}^{\mathcal{Y}}_{0,\infty}})
\end{align*}
Likewise, $\mathcal{F}^{\mathcal{Y}}_{0,n}$ is an increasing sequence of sigma fields, therefore by \cite[Theorem 3]{BarronSorrento} we have that if the relative entropy is ever finite, the relative entropy restricted to these sigma fields is the relative entropy over the limit field, that is
\begin{align*}
& \lim_{n \to \infty}D(P^{\mu}|_{\mathcal{F}^{\mathcal{Y}}_{0,n}}\|P^{\nu}|_{\mathcal{F}^{\mathcal{Y}}_{0,n}}) =D(P^{\mu}|_{\mathcal{F}^{\mathcal{Y}}_{0,\infty}}\|P^{\nu}|_{\mathcal{F}^{\mathcal{Y}}_{0,\infty}})
\end{align*}
Therefore, 
\begin{align*}
& \lim_{n \to \infty} E^{\mu}[D(\pi_{n}^{\mu}\|\pi_{n}^{\nu})] \\
&=D(P^{\mu}|_{\bigcap_{n \geq 0}\mathcal{F}^{\mathcal{X}}_{n,\infty} \vee \mathcal{F}^{\mathcal{Y}}_{0,\infty}}\|P^{\nu}|_{\bigcap_{n \geq 0}\mathcal{F}^{\mathcal{X}}_{n,\infty} \vee \mathcal{F}^{\mathcal{Y}}_{0,\infty}}) \\
& -D(P^{\mu}|_{\mathcal{F}^{\mathcal{Y}}_{0,\infty}}\|P^{\nu}|_{\mathcal{F}^{\mathcal{Y}}_{0,\infty}})
\end{align*}
By Lemma \ref{restrict_structure} we have
\begin{align*}
\frac{dP^{\mu}|_{\bigcap_{n \geq 0}\mathcal{F}^{\mathcal{X}}_{n,\infty} \vee \mathcal{F}^{\mathcal{Y}}_{0,\infty}}}{dP^{\nu}|_{\bigcap_{n \geq 0}\mathcal{F}^{\mathcal{X}}_{n,\infty} \vee \mathcal{F}^{\mathcal{Y}}_{0,\infty}}}&=E^{\nu}\left[\left.\frac{d\mu}{d\nu}(X_{0})\right|\bigcap_{n \geq 0}\mathcal{F}^{\mathcal{X}}_{n,\infty} \vee \mathcal{F}^{\mathcal{Y}}_{0,\infty}\right]=f_{1}\\
\frac{dP^{\mu}|_{\mathcal{F}^{\mathcal{Y}}_{0,\infty}}}{d P^{\nu}|_{\mathcal{F}^{\mathcal{Y}}_{0,\infty}}}&=E^{\nu}\left[\left.\frac{d\mu}{d\nu}(X_{0})\right|\mathcal{F}^{\mathcal{Y}}_{0,\infty}\right]=f_{2}
\end{align*}
Note that $f_{1}$ is $\bigcap_{n \geq 0}\mathcal{F}^{\mathcal{X}}_{n,\infty} \vee \mathcal{F}^{\mathcal{Y}}_{0,\infty}$ measurable, while $f_{2}$ is $ \mathcal{F}^{\mathcal{Y}}_{0,\infty}$ measurable, and $\mathcal{F}^{\mathcal{Y}}_{0,\infty}\subset \bigcap_{n \geq 0}\mathcal{F}^{\mathcal{X}}_{n,\infty} \vee \mathcal{F}^{\mathcal{Y}}_{0,\infty}$. By Lemma \ref{filt_iff}, we have that if the filter merges in total variation in expectation, then for a set of state and observation sequences $\omega=(x_{i},y_{i})_{i=0}^{\infty} \in A \subset \mathcal{F}^{\mathcal{X}}_{0,\infty} \vee \mathcal{F}^{\mathcal{Y}}_{0,\infty}$ with $P^{\nu}(A)=1$, we have $f_{1}(\omega)=f_{2}(\omega)$. Yet this then means over the set $A$ of $P^{\nu}$ measure $1$, $f_{1}=f_{2}$ is $\mathcal{F}^{\mathcal{Y}}_{0,\infty}$ measurable. We then have
\begin{align*}
&D(P^{\mu}|_{\bigcap_{n \geq 0}\mathcal{F}^{\mathcal{X}}_{n,\infty} \vee \mathcal{F}^{\mathcal{Y}}_{0,\infty}}\|P^{\nu}|_{\bigcap_{n \geq 0}\mathcal{F}^{\mathcal{X}}_{n,\infty} \vee \mathcal{F}^{\mathcal{Y}}_{0,\infty}}) \\
& -D(P^{\mu}|_{\mathcal{F}^{\mathcal{Y}}_{0,\infty}}\|P^{\nu}|_{\mathcal{F}^{\mathcal{Y}}_{0,\infty}})\\
&=E^{\mu}[\log(f_{1})]-E^{\mu}[\log(f_{2})]=E^{\nu}[f_{1}\log(f_{1})]-E^{\nu}[f_{2}\log(f_{2})]\\
&=\int_{\Omega} f_{1}(\omega)\log(f_{1}(\omega))dP^{\nu}|_{\bigcap_{n \geq 0}\mathcal{F}^{\mathcal{X}}_{n,\infty} \vee \mathcal{F}^{\mathcal{Y}}_{0,\infty}}(\omega) \\
& -\int_{\Omega}f_{2}(\omega)\log(f_{2}(\omega))dP^{\nu}|_{\mathcal{F}^{\mathcal{Y}}_{0,\infty}}(\omega)\\
&=\int_{A} f_{1}(\omega)\log(f_{1}(\omega))dP^{\nu}|_{\bigcap_{n \geq 0}\mathcal{F}^{\mathcal{X}}_{n,\infty} \vee \mathcal{F}^{\mathcal{Y}}_{0,\infty}}(\omega) \\
&-\int_{A}f_{2}(\omega)\log(f_{2}(\omega))dP^{\nu}|_{\mathcal{F}^{\mathcal{Y}}_{0,\infty}}(\omega)\\
&=\int_{A} f_{1}(\omega)\log(f_{1}(\omega))dP^{\nu}|_{\mathcal{F}^{\mathcal{Y}}_{0,\infty}}(\omega) \\
& -\int_{A}f_{2}(\omega)\log(f_{2}(\omega))dP^{\nu}|_{\mathcal{F}^{\mathcal{Y}}_{0,\infty}}(\omega) =0
\end{align*}
 Therefore, if the relative entropy of the filter is ever finite, then total variation merging in expectation is equivalent to merging in relative entropy.
\end{proof}

\section{Conclusion}\label{conclusions}
We presented a notion of stochastic observability for non-linear systems. This notion is explicit, is relatively easily computed due to its functional approximation formulation, and is shown via examples to be applicable to a large class of systems. The implications of this definition for filter stability were presented in detail. Further relations under various stability criteria and implications were studied.

 \section{Appendix. Supporting Results for Section \ref{prooftv_equal_placeholder}}\label{appenSupTV}

We present a number of supporting results. The approach for these build on similar arguments in \cite{chigansky2009intrinsic} and \cite{van2009stability}. The proofs here are kept brief due to space constraints or omitted.

\begin{lemma}\label{restrict_structure}
Assume $\mu \ll \nu$. For any sigma field $\mathcal{G} \subseteq \mathcal{F}^{\mathcal{X}}_{0,\infty} \vee \mathcal{F}^{\mathcal{Y}}_{0,\infty}$ we have:
\begin{align*}
\frac{dP^{\mu}|_{\mathcal{G}}}{dP^{\nu}|_{\mathcal{G}}}&=E^{\nu}\left[\left.\frac{d\mu}{d\nu}(X_{0})\right|\mathcal{G}\right]~~~P^{\mu}~a.s.
\end{align*}
\end{lemma}

\begin{lemma}\label{condition_structure}
Assume $\mu \ll \nu$. For any two sigma fields $\mathcal{G}_{1},\mathcal{G}_{2} \subset \mathcal{F}^{\mathcal{X}}_{0,\infty} \vee \mathcal{F}^{\mathcal{Y}}_{0,\infty}$, let $P^{\mu}|_{\mathcal{G}_{1}}|\mathcal{G}_{2}$ represent the probability measure $P^{\mu}$ restricted to $\mathcal{G}_{1}$, conditioned on field $\mathcal{G}_{2}$. We then have
\begin{align*}
&\frac{dP^{\mu}|_{\mathcal{G}_{1}}|\mathcal{G}_{2}}{dP^{\nu}|_{\mathcal{G}_{1}}|\mathcal{G}_{2}} =\frac{E^{\nu}[\frac{d\mu}{d\nu}(X_{0})|\mathcal{G}_{1} \vee \mathcal{G}_{2}]}{E^{\nu}[\frac{d\mu}{d\nu}(X_{0})|\mathcal{G}_{2}]}~~~P^{\mu}~a.s.
\end{align*}
\end{lemma}

\begin{lemma}\label{tv_structure}
Assume $\mu \ll \nu$, for any two sigma fields $\mathcal{G}_{1},\mathcal{G}_{2} \subset \mathcal{F}^{\mathcal{X}}_{0,\infty} \vee \mathcal{F}^{\mathcal{Y}}_{0,\infty}$ we have $P^{\mu}$~a.s.
\begin{align*}
& \|P^{\mu}|_{\mathcal{G}_{1}}|\mathcal{G}_{2}-P^{\nu}|_{\mathcal{G}_{1}}|\mathcal{G}_{2}\|_{TV} \\
&=\frac{E^{\nu}\left[\left.~\left|E^{\nu}\left[\frac{d\mu}{d\nu}(X_{0})|\mathcal{G}_{1} \vee \mathcal{G}_{2} \right]-E^{\nu}\left[\frac{d\mu}{d\nu}(X_{0})|\mathcal{G}_{2} \right] \right|~\right|\mathcal{G}_{2} \right]}{E^{\nu}\left[\frac{d\mu}{d\nu}(X_{0})|\mathcal{G}_{2} \right]}
\end{align*}
\end{lemma}

For the specific case of the non-linear filter, that is $\mathcal{G}_{1}=\mathcal{F}^{\mathcal{X}}_{n}$ and $\mathcal{G}_{2}=\mathcal{F}^{\mathcal{Y}}_{0,n}$, the results presented above imply the following known results in the literature.
\begin{lemma}
\label{filter_rd_derivative}\cite[Lemma 5.6]{van2009stability}
Assume $\mu \ll \nu$. Then we have that $\pi_{n}^{\mu} \ll \pi_{n}^{\nu}$ a.s. and we have
\begin{align}\label{filt_rd_derivative}
\frac{d\pi_{n}^{\mu}}{d\pi_{n}^{\nu}}(x)&=\frac{E^{\nu}[\frac{d\mu}{d\nu}(X_{0})|Y_{[0,n]},X_{n}=x]}{E^{\nu}[\frac{d\mu}{d\nu}(X_{0})|Y_{[0,n]}]}~~~~~P^{\mu}~a.s.
\end{align}
\end{lemma}

%

%

\begin{lemma}\cite[Equation 1.10]{chigansky2009intrinsic}\label{filt_iff}
The filter merges in total variation in expectation if and only if $P^{\nu}$ a.s.
\begin{align}\label{sigma_field_condition}
E^{\nu}\left[\left.\frac{d\mu}{d\nu}(X_{0})\right|\bigcap_{n \geq 0}\mathcal{F}_{0,\infty}^{\mathcal{Y}}\vee \mathcal{F}_{n,\infty}^{\mathcal{X}}\right]&=E^{\nu}\left[\left.\frac{d\mu}{d\nu}(X_{0})\right|{F}_{0,\infty}^{\mathcal{Y}}\right|
\end{align} 
\end{lemma}

Since our results apply to any general sigma field, not just the fields used in the analysis of the filter, we can study the predictor process to establish Lemmas \ref{predictor_rd_derivative}, \ref{tv_expression_pred}, and \ref{pred_sigma_field_condition}, in the following.

\begin{lemma} \label{predictor_rd_derivative}
Assume $\mu \ll \nu$. Then we have that $\pi_{n-}^{\mu}\ll \pi_{n-}^{\mu}$ $P^{\mu}$ a.s. and we have
\begin{align}
\frac{d\pi_{n-}^{\mu}}{d\pi_{n-}^{\nu}}(x)&=\frac{E^{\nu}[\frac{d\mu}{d\nu}(X_{0})|Y_{[0,n-1]},X_{n}=x]}{E^{\nu}[\frac{d\mu}{d\nu}(X_{0})|Y_{[0,n-1]}]}~~~~~P^{\mu}~a.s.\label{pred_rd_derivative}
\end{align}
\end{lemma}

\begin{proof}
These results become clear from Lemma \ref{condition_structure} when we state the predictor as $P^{\mu}$ restricted to $\mathcal{F}^{\mathcal{X}}_{n}$ conditioned on $\mathcal{F}^{\mathcal{Y}}_{0,n-1}$
\end{proof}

\begin{lemma}\label{tv_expression_pred}
Assume $\mu \ll \gamma$  for some measure $\gamma$. We can express
\begin{align}
& \|\pi_{n-}^{\mu}-\pi_{n-}^{\gamma}\|_{TV}= \nonumber \\
&\frac{E^{\gamma}\left[\left.\left|E^{\gamma}[\frac{d\mu}{d\gamma}(X_{0})|Y_{[0,\infty)},X_{[n,\infty)}]-E^{\gamma}[\frac{d\mu}{d\gamma}(X_{0})|Y_{[0,n-1]}]\right|\right|Y_{[0,n-1]}\right]}{E^{\gamma}\left[\left.\frac{d\mu}{d\gamma}(X_{0})\right|Y_{[0,n-1]}\right]}\label{tv_pred_as_martingale}
\end{align}
\end{lemma}

\begin{proof}
By Lemma \ref{tv_structure} we can write
\begin{align*}
& \|\pi_{n-}^{\mu}-\pi_{n-}^{\gamma}\|_{TV}= \\
&\frac{E^{\gamma}\left[\left.|E^{\gamma}[\frac{d\mu}{d\gamma}(X_{0})|Y_{[0,n-1]},X_{n}]-E^{\gamma}[\frac{d\mu}{d\gamma}(X_{0})|Y_{[0,n-1]}]|\right|Y_{[0,n-1]}\right]}{E^{\gamma}\left[\left.\frac{d\mu}{d\gamma}(X_{0})\right|Y_{[0,n-1]}\right]}
\end{align*}

Since $Y_{n}$ is a function of $X_{n}$ and the random noise $Z_{n}$ which is independent of $X_{n}$ and past $Y_{[0,n-1]}$ measurements, we have that $\sigma(Y_{[0,n-1]},X_{n})=\sigma(Y_{[0,n]},X_{n})$. Further, by the Markov property we have that we have that $(X_{[0,n-1]},Y_{[0,n-1]})$ are independent of $(X_{[n+1,\infty)},Y_{[n+1,\infty)})$ conditioned on $(X_{n},Y_{n})$ therefore we can state
\begin{align*}
E^{\gamma}[\frac{d\mu}{d\nu}(X_{0})|Y_{[0,n-1]},X_{n}]&=E^{\gamma}[\frac{d\mu}{d\nu}(X_{0})|Y_{[0,\infty)},X_{[n,\infty)}]
\end{align*}
\end{proof}

\begin{lemma}\label{pred_sigma_field_condition}
The predictor merges in total variation in expectation if and only if
\begin{align}
E^{\nu}\left[\left.\frac{d\mu}{d\nu}(X_{0})\right|\bigcap_{n \geq 1}\mathcal{F}_{0,\infty}^{\mathcal{Y}}\vee \mathcal{F}_{n,\infty}^{\mathcal{X}}\right]&=E^{\nu}\left[\left.\frac{d\mu}{d\nu}(X_{0})\right|{F}_{0,\infty}^{\mathcal{Y}}\right]~P^{\nu}~a.s.\label{pred_sigma_eqn}
\end{align} 
\end{lemma}
\begin{proof}
Building on the proof of Lemma \ref{tv_expression_pred}, we have
\begin{align*}
&E^{\mu}\left[\|\pi_{n-}^{\mu}- \pi_{n-}^{\nu}\|_{TV} \right]=E^{\nu}\left[\frac{dP^{\mu}|_{\mathcal{F}^{\mathcal{Y}}_{0,n-1}}}{dP^{\nu}|_{\mathcal{F}^{\mathcal{Y}}_{0,n-1}}}\|\pi_{n-}^{\mu}- \pi_{n-}^{\nu}\|_{TV} \right]\\
&=E^{\nu}\left[E^{\nu}\left[\left.\frac{d\mu}{d\nu}(X_{0})\right|Y_{[0,n-1]}\right]\|\pi_{n-}^{\mu}- \pi_{n-}^{\nu}\|_{TV} \right]\\
&=E^{\nu} [E^{\nu}[|E^{\nu}[\frac{d\mu}{d\nu}(X_{0})|Y_{[0,\infty)},X_{[n,\infty)}] \\
& \quad \quad -E^{\nu}[\frac{d\mu}{d\nu}(X_{0})|Y_{[0,n-1]}]| |Y_{[0,n-1]} ] ]\\
&=E^{\nu}\left[|E^{\nu}[\frac{d\mu}{d\nu}(X_{0})|Y_{[0,\infty)},X_{[n,\infty)}]-E^{\nu}[\frac{d\mu}{d\nu}(X_{0})|Y_{[0,n-1]}]|\right]
\end{align*}
We then see that $A_{n}=E^{\nu}[\frac{d\mu}{d\nu}(X_{0})|Y_{[0,n-1]}]$ is a non-negative  uniformly integrable martingale adapted to the increasing filtration $\mathcal{F}^{\mathcal{Y}}_{0,n-1}$. Hence the limit as $n \to \infty$ in $L^{1}(P^{\nu})$ is $E^{\nu}[\frac{d\mu}{d\nu}(X_{0})|\mathcal{F}^{\mathcal{Y}}_{0,\infty}]$. Similarly, we can view $B_{n}=E^{\nu}[\frac{d\mu}{d\nu}(X_{0})|Y_{[0,\infty)},X_{[n,\infty)}]$ as backwards non-negative uniformly integrable martingale with respect to the decreasing sequence of filtrations $\mathcal{F}_{0,\infty}^{\mathcal{Y}}\vee \mathcal{F}_{n,\infty}^{\mathcal{X}}$. Then by the backwards martingale convergence theorem, the limit as $n \to \infty$ in $L^{1}(P^{\nu})$ is $E^{\nu}[\frac{d\mu}{d\nu}(X_{0})|\bigcap_{n=1}^{\infty}\mathcal{F}^{\mathcal{Y}}_{1,\infty}\vee \mathcal{F}^{\mathcal{X}}_{n,\infty}]$. It is then clear the total variation in expectation is zero if and only if equation (\ref{pred_sigma_eqn}) holds.
\end{proof}


\begin{thebibliography}{10}

\bibitem{anderson1981detectability}
B.~D.~O. Anderson and J.~B. Moore.
\newblock Detectability and stabilizability of time-varying discrete-time
  linear systems.
\newblock {\em SIAM Journal on Control and Optimization}, 19(1):20--32, 1981.

\bibitem{bartse74}
Y.~Bar-Shalom and E.~Tse.
\newblock Dual effect, certainty equivalence, and separation in stochastic
  control.
\newblock {\em IEEE Transactions on Automatic Control}, 19(5):494--500, October
  1974.

\bibitem{BarronInfoMartingales}
A.~R. Barron.
\newblock Information theory and martingales.
\newblock In {\em IEEE International Symposium on Information Theory, recent
  results session}, Budapest, Hungary, 1991.

\bibitem{BarronSorrento}
A.~R. Barron.
\newblock Limits of information, {M}arkov chains, and projections.
\newblock Sorrento, Italy, 2000. Proceedings of the IEEE Int. Symp. on Inform.
  Theory p. 25.

\bibitem{Bil86}
P.~Billingsley.
\newblock {\em Probability and Measure}.
\newblock Wiley, \newblock New York, 2nd edition, 1986.

\bibitem{Bil99}
P.~Billingsley.
\newblock {\em Convergence of probability measures}.
\newblock New York: Wiley, 2nd edition, 1999.

\bibitem{blackwell1962merging}
D.~Blackwell and L.~Dubins.
\newblock Merging of opinions with increasing information.
\newblock {\em The Annals of Mathematical Statistics}, 33(3):882--886, 1962.

\bibitem{Bogachev}
V.~I. Bogachev.
\newblock {\em Measure Theory}.
\newblock Springer-Verlag, Berlin, 2007.

\bibitem{BorkarRealization}
V.~S. Borkar.
\newblock White-noise representations in stochastic realization theory.
\newblock {\em SIAM J. on Control and Optimization}, 31:1093--1102, 1993.

\bibitem{Caines}
P.~E. Caines.
\newblock {\em Linear Stochastic Systems}.
\newblock John Wiley \& Sons, New York, NY, 1988.

\bibitem{Chen}
C.~T. Chen.
\newblock {\em Linear Systems Theory and Design}.
\newblock Oxford University Press, Oxford, 1999.

\bibitem{chigansky2010complete}
P.~Chigansky and R.~Van Handel.
\newblock A complete solution to {B}lackwell's unique ergodicity problem for
  hidden {M}arkov chains.
\newblock {\em The Annals of Applied Probability}, 20(6):2318--2345, 2010.

\bibitem{chigansky2006role}
P.~Chigansky and R.~Liptser.
\newblock On a role of predictor in the filtering stability.
\newblock {\em Electron. Comm. Probab}, 11:129--140, 2006.

\bibitem{chigansky2009intrinsic}
P.~Chigansky, R.~Liptser, and R.~van Handel.
\newblock Intrinsic methods in filter stability.
\newblock {\em Handbook of Nonlinear Filtering}, 2009.

\bibitem{clark1999relative}
J.M.C. Clark, D.~L. Ocone, and C.~Coumarbatch.
\newblock Relative entropy and error bounds for filtering of markov processes.
\newblock {\em Mathematics of Control, Signals and Systems}, 12(4):346--360,
  1999.

\bibitem{csiszar1967information}
I.~Csisz{\'a}r.
\newblock Information-type measures of difference of probability distributions
  and indirect observation.
\newblock {\em studia scientiarum Mathematicarum Hungarica}, 2:229--318, 1967.

\bibitem{d1988merging}
A.~D'Aristotile, P.~Diaconis, and D.~Freedman.
\newblock On merging of probabilities.
\newblock {\em Sankhy{\=a}: The Indian Journal of Statistics, Series A}, pages
  363--380, 1988.

\bibitem{Dud02}
R.~M. Dudley.
\newblock {\em Real Analysis and Probability}.
\newblock Cambridge University Press, Cambridge, 2nd edition, 2002.

\bibitem{dugundji}
J.~Dugundji.
\newblock An extension of tietze's theorem.
\newblock {\em Pacific Journal of Mathematics}, 1(3):353--367, 1951.

\bibitem{ethier2009markov}
S.~N. Ethier and T.~G. Kurtz.
\newblock {\em Markov processes: characterization and convergence}, volume 282.
\newblock John Wiley \& Sons, 2009.

\bibitem{FeKaZg14}
E.A. Feinberg, P.O. Kasyanov, and M.Z. Zgurovsky.
\newblock Partially observable total-cost {M}arkov decision process with weakly
  continuous transition probabilities.
\newblock {\em Mathematics of Operations Research}, 41(2):656--681, 2016.

\bibitem{Fritz73}
J.~Fritz.
\newblock An information-theoretical proof of limit theorems for reversible
  {M}arkov processes.
\newblock In {\em In Trans. Sixth Prague Conf. on Inform. Theory, Statist.
  Decision Functions, Random Processes}, 1973.

\bibitem{gihman2012controlled}
I.~I. Gihman and A.~V. Skorohod.
\newblock {\em Controlled stochastic processes}.
\newblock Springer Science \& Business Media, 2012.

\bibitem{golowich2022planning}
N.~Golowich, A.~Moitra, and D.~Rohatgi.
\newblock Planning in observable pomdps in quasipolynomial time.
\newblock {\em arXiv preprint arXiv:2201.04735}, 2022.

\bibitem{gray2011entropy}
R~.M. Gray.
\newblock {\em Entropy and information theory}.
\newblock Springer Science \& Business Media, 2011.

\bibitem{HarremosHolst}
P.~Harremo\"es and K.~K. Holst.
\newblock Convergence of {M}arkov chains in information divergence.
\newblock {\em J. Theoretical Probability}, 22:186--202, 2009.

\bibitem{harremoes2010thinning}
P.~Harremo{\"e}s, O.~Johnson, and I.~Kontoyiannis.
\newblock Thinning, entropy, and the law of thin numbers.
\newblock {\em IEEE Transactions on Information Theory}, 56(9):4228--4244,
  2010.

\bibitem{hermann1977nonlinear}
R.~Hermann and A.~Krener.
\newblock Nonlinear controllability and observability.
\newblock {\em IEEE Transactions on Automatic Control}, 22(5):728--740, 1977.

\bibitem{white1994finite}
C.~C.~White III and W.~T. Scherer.
\newblock Finite-memory suboptimal design for partially observed markov
  decision processes.
\newblock {\em Operations Research}, 42(3):439--455, 1994.

\bibitem{Kalman}
R.~E. Kalman.
\newblock A new approach to linear filtering and prediction problems.
\newblock {\em Transaction of the ASME J. of Basic Engineering}, pages 35--45,
  March 1960.

\bibitem{KSYWeakFellerSysCont}
A.D Kara, N.~Saldi, and S.~Y\"uksel.
\newblock Weak {F}eller property of non-linear filters.
\newblock {\em Systems \& Control Letters}, 134:104--512, 2019.

\bibitem{KYSICONPrior}
A.D Kara and S.~Y\"uksel.
\newblock Robustness to incorrect priors in partially observed stochastic
  control.
\newblock {\em SIAM Journal on Control and Optimization}, 57(3):1929--1964,
  2019.

\bibitem{kara2021convergence}
A.D Kara and S.~Y\"uksel.
\newblock Convergence of finite memory {Q}-learning for {POMDP}s and near
  optimality of learned policies under filter stability.
\newblock {\em Mathematics of Operations Research (also arXiv:2103.12158)},
  2022.

\bibitem{kara2020near}
A.D Kara and S.~Y\"uksel.
\newblock Near optimality of finite memory feedback policies in partially
  observed markov decision processes.
\newblock {\em Journal of Machine Learning Research}, 23(11):1--46, 2022.

\bibitem{kim2022duality}
J.-W. Kim and P.~G. Mehta.
\newblock Duality for nonlinear filtering i: Observability.
\newblock {\em arXiv preprint arXiv:2208.06586}, 2022.

\bibitem{kim2019lagrangian}
J.-W. Kim, P.~G. Mehta, and S.~P. Meyn.
\newblock What is the {L}agrangian for nonlinear filtering?
\newblock In {\em 2019 IEEE 58th Conference on Decision and Control (CDC)},
  pages 1607--1614. IEEE, 2019.

\bibitem{kunita1971asymptotic}
H.~Kunita.
\newblock Asymptotic behavior of the nonlinear filtering errors of markov
  processes.
\newblock {\em Journal of Multivariate Analysis}, 1(4):365--393, 1971.

\bibitem{kushner2014partial}
H.~J. Kushner.
\newblock A partial history of the early development of continuous-time
  nonlinear stochastic systems theory.
\newblock {\em Automatica}, 50(2):303--334, 2014.

\bibitem{KushnerKalmanFilter}
H.J. Kushner.
\newblock {\em Introduction to Stochastic Control Theory}.
\newblock Holt, Rinehart and Winston, New York, 1972.

\bibitem{linder2014optimal}
T.~Linder and S.~Y{\"u}ksel.
\newblock On optimal zero-delay coding of vector markov sources.
\newblock {\em IEEE Trans. Information Theory}, 60(10):5975--5991, 2014.

\bibitem{liu2011stochastic}
A.~R. Liu.
\newblock {\em Stochastic observability, reconstructibility, controllability,
  and reachability}.
\newblock PhD thesis, UC San Diego, 2011.

\bibitem{liu2010observability}
A.R. Liu and R.R. Bitmead.
\newblock Observability and reconstructibility of hidden markov models:
  Implications for control and network congestion control.
\newblock In {\em 49th IEEE Conference on Decision and Control (CDC)}, pages
  918--923. IEEE, 2010.

\bibitem{MYRobustControlledFS}
C.~McDonald and S.~Y\"uksel.
\newblock Robustness to incorrect priors and controlled filter stability in
  partially observed stochastic control.
\newblock {\em SIAM Journal on Control and Optimization}, 60(2):842--870, 2022.

\bibitem{MeynBook}
S.~P. Meyn and R.~Tweedie.
\newblock {\em {M}arkov Chains and Stochastic Stability}.
\newblock Springer-Verlag, London, 1993.

\bibitem{nijmeijer1982observability}
H.~Nijmeijer.
\newblock Observability of autonomous discrete time non-linear systems: a
  geometric approach.
\newblock {\em International journal of control}, 36(5):867--874, 1982.

\bibitem{ocone1999asymptotic}
D.~L. Ocone.
\newblock Asymptotic stability of bene{\v{s}} filters.
\newblock {\em Stochastic analysis and applications}, 17(6):1053--1074, 1999.

\bibitem{reddy2019stability}
A.S. Reddy and A.~Apte.
\newblock Stability of non-linear filter for deterministic dynamics.
\newblock {\em arXiv preprint arXiv:1910.14348}, 2019.

\bibitem{rudin2006real}
W.~Rudin.
\newblock {\em Real and Complex Analysis}.
\newblock Tata McGraw-Hill Education, 2006.

\bibitem{sontag1984concept}
E.~D. Sontag.
\newblock A concept of local observability.
\newblock {\em Systems \& Control Letters}, 5(1):41--47, 1984.

\bibitem{sowers1992discrete}
R.B Sowers and A.M. Makowski.
\newblock Discrete-time filtering for linear systems with non-{G}aussian
  initial conditions: asymptotic behavior of the difference between the {MMSE}
  and {LMSE} estimates.
\newblock {\em IEEE transactions on automatic control}, 37(1):114--120, 1992.

\bibitem{ugrinovskii2003observability}
V.A. Ugrinovskii.
\newblock Observability of linear stochastic uncertain systems.
\newblock {\em IEEE Transactions on Automatic Control}, 48(12):2264--2269,
  2003.

\bibitem{van2008discrete}
R.~van Handel.
\newblock Discrete time nonlinear filters with informative observations are
  stable.
\newblock {\em Electronic Communications in Probability}, 13:562--575, 2008.

\bibitem{van2009observability}
R.~van Handel.
\newblock Observability and nonlinear filtering.
\newblock {\em Probability theory and related fields}, 145(1-2):35--74, 2009.

\bibitem{van2009stability}
R.~van Handel.
\newblock The stability of conditional markov processes and markov chains in
  random environments.
\newblock {\em The Annals of Probability}, 37(5):1876--1925, 2009.

\bibitem{van2009uniform}
R.~van Handel.
\newblock Uniform observability of hidden {M}arkov models and filter stability
  for unstable signals.
\newblock {\em The Annals of Applied Probability}, 19(3):1172--1199, 2009.

\bibitem{van2010nonlinear}
R.~van Handel.
\newblock Nonlinear filtering and systems theory.
\newblock In {\em Proceedings of the 19th International Symposium on
  Mathematical Theory of Networks and Systems (MTNS semi-plenary paper)}, 2010.

\end{thebibliography}
\end{document}